\documentclass{article}

\usepackage{amsmath,amssymb,amsthm}
\usepackage[colorlinks=true,urlcolor=blue,linkcolor=blue,plainpages=false,pdfpagelabels
]{hyperref}

\newcommand{\details}[1]{$\,$\\***DETAILS {\bf #1}***\\}
\renewcommand{\details}[1]{}

\newtheorem{definition}{Definition}[section]
\newtheorem{proposition}[definition]{Proposition}
\newtheorem{remark}[definition]{Remark}
\newtheorem{theorem}[definition]{Theorem}
\newtheorem{corollary}[definition]{Corollary}
\newtheorem{lemma}[definition]{Lemma}

\newtheorem{example}[definition]{Example}

\def\RR{{\mathbb R}}
\def\NN{{\mathbb N}}

\def\N{{\mathbb N}}
\def\R{{\mathbb R}}

\newcommand{\cT}{\mathcal{T}}

\newcommand{\eps}{\varepsilon}

\newcommand{\ba}{\begin{array}} \newcommand{\ea}{\end{array}}
\newcommand {\bea} {\begin{eqnarray}} \newcommand {\eea} {\end {eqnarray}}
\newcommand{\beq}{\begin{equation}} \newcommand{\eeq}{\end{equation}}
\newcommand{\be}{\begin{enumerate}} \newcommand{\ee}{\end{enumerate}}
\newcommand {\bua} {\begin{eqnarray*}}
\newcommand {\eua} {\end {eqnarray*}}

\newcounter{ct}

\newcommand{\ds}{\displaystyle}

\newcommand{\ra}{\rightarrow}
\newcommand{\Ra}{\Rightarrow}

\newcommand{\se}{\subseteq}

\newcommand{\ol}{\overline}

\newcommand{\limn}{\ds\lim_{n\to\infty}}
\newcommand{\linfn}{\ds\liminf_{n\to\infty}}

\let\OLDthebibliography\thebibliography
\renewcommand\thebibliography[1]{
  \OLDthebibliography{#1}
  \setlength{\parskip}{0pt}
  \setlength{\itemsep}{0pt plus 0.3ex}
}

\begin{document}

\title{Quantitative results on Fej\'er monotone sequences}
\author{Ulrich Kohlenbach${}^1$, Lauren\c{t}iu Leu\c{s}tean${}^{2,3}$,  Adriana Nicolae${}^{4,5}$ \\[0.2cm]
\footnotesize ${}^1$ Department of Mathematics, Technische Universit\" at Darmstadt,\\
\footnotesize Schlossgartenstra\ss{}e 7, 64289 Darmstadt, Germany\\[0.1cm]
\footnotesize ${}^2$ Faculty of Mathematics and Computer Science, University of Bucharest,\\
\footnotesize Academiei 14,  P.O. Box 010014, Bucharest, Romania\\[0.1cm]
\footnotesize ${}^3$ Simion Stoilow Institute of Mathematics of the Romanian Academy,\\
\footnotesize P. O. Box 1-764, 014700 Bucharest, Romania\\[0.1cm]
\footnotesize ${}^4$ Department of Mathematics, Babe\c{s}-Bolyai University, \\
\footnotesize  Kog\u{a}lniceanu 1, 400084 Cluj-Napoca, Romania\\[0.1cm]
\footnotesize ${}^5$ Simion Stoilow Institute of Mathematics of the Romanian Academy, \\
\footnotesize Research group of the project PD-3-0152,\\
\footnotesize P. O. Box 1-764, 014700 Bucharest, Romania\\[0.1cm]
\footnotesize E-mails: kohlenbach@mathematik.tu-darmstadt.de, Laurentiu.Leustean@imar.ro, \\
\footnotesize \!\!\!\!\!\!\ anicolae@math.ubbcluj.ro
}

\date{}

\maketitle

\begin{abstract}
We provide in a unified way quantitative forms of 
strong convergence results for numerous iterative procedures which 
satisfy a general type of Fej\'er monotonicity where the convergence 
uses the compactness of the underlying set. These quantitative versions 
are in the form of 
explicit rates of so-called metastability in the sense of T. Tao.
Our approach covers examples ranging from the proximal 
point algorithm for maximal monotone operators to various fixed point 
iterations $(x_n)$ for firmly nonexpansive, asymptotically nonexpansive, 
strictly pseudo-contractive and other types of mappings. Many of the results 
hold in a general metric setting with some convexity structure added 
(so-called $W$-hyperbolic spaces). Sometimes uniform convexity is assumed still covering the important class of CAT(0)-spaces due to Gromov.
\end{abstract}

\section{Introduction}
This paper provides in a unified way quantitative forms of 
strong convergence results for numerous iterative procedures which 
satisfy a general type of Fej\'er monotonicity where the convergence 
uses the compactness of the underlying set. Fej\'er monotonicity is 
a key notion employed in the study of many problems in convex optimization and programming, 
fixed point theory and the study of (ill-posed) inverse problems 
(see e.g. \cite{Vasin,Combettes}). These quantitative forms have been 
obtained using the logic-based proof mining approach (as developed 
e.g. in \cite{Kohlenbach(book)}) 
but the results are presented here in a way which avoids any 
explicit reference to notions or tools from logic. \\ 
Our approach covers examples ranging from the proximal 
point algorithm for maximal monotone operators to various fixed point 
iterations $(x_n)$ for firmly nonexpansive, asymptotically nonexpansive, 
strictly pseudo-contractive and other types of mappings. Many of the results 
hold in a general metric setting with some convexity structure added 
(so-called $W$-hyperbolic spaces in the sense of \cite{Koh05a}). 
Sometimes uniform convexity is assumed still covering Gromov's 
CAT(0)-spaces. \\[1mm] 
For reasons from computability theory, effective rates of convergence for 
$(x_n)$ in $X$ are usually ruled out even when the space $X$ 
in question and the map $T$ used in the iteration are effective: usually 
$(x_n)$ will converge to a fixed point of $T$ but in general $T$ will not
possess a computable fixed point and even when it does (e.g. when $X$ is 
$\R^n$ and the fixed point set is convex) the usual iterations will not 
converge to a computable point and hence will not converge with an effective 
rate of convergence (see \cite{Neumann} for details on all this). \\[1mm]
The Cauchy property of $(x_n)$ can, however, be reformulated in the 
equivalent form 
\[ (*)\quad \forall k\in\N\,\forall g:\N\to\N \,\exists N\in\N\,\forall i,j\in 
[N,N+g(N)]\ \left( d(x_i,x_j) \le\frac{1}{k+1}\right)\]
and for this form, highly uniform computable bounds $\exists N\le\Phi(k,g)$ 
on $\exists N$ can be obtained. $(*)$ is known in mathematical logic 
since 1930 as Herbrand normal form and bounds $\Phi$ have been studied in 
the so-called Kreisel no-counterexample interpretation (which in turn is a 
special case of the G\"odel functional interpretation) since the 50's 
(see \cite{Kohlenbach(book)}).
More recently, $(*)$ has been made popular under the name of `metastability' 
by Terence Tao, who used the existence of uniform bounds on $N$ in 
the context of ergodic theory (\cite{Tao07,Tao08}. Moreover, 
Walsh \cite{Wal12} used again metastability 
to show the  $L^2$-convergence of multiple polynomial ergodic averages arising from 
nilpotent groups of measure-preserving transformations. 
\\[1mm] In nonlinear analysis, rates of metastability 
$\Phi$  
for strong convergence results of nonlinear iterations have been first 
considered and extracted in \cite{KohLam04,Koh05} (and in many 
other cases since then). The point of departure of our investigation is \cite{Koh05} which uses Fej\'er
monotonicity and where some of the arguments of the present paper have first been
used in a special context. \\[1mm] 
Let $F\subseteq X$ be a subset 
of $X$ and recall that $(x_n)$ is Fej\'er monotone w.r.t. $F$ if 
\[ (+) \quad d(x_{n+1},p)\le d(x_n,p), \ \mbox{for all $n\in\N$ and $p\in F$}. \]
We think of $F$ as being the intersection $ F=\bigcap_{k\in\N} AF_k$ 
of approximations $AF_{k+1}\subseteq AF_k\subseteq X$ to $F,$ one 
prime example being $F:=Fix(T)$ and $AF_k:=\{ p\in X\,|\,d(p,Tp)
\le 1/(k+1)\},$ where $Fix(T)$ denotes the fixed point set of some 
selfmap $T:X\to X.$ \\[1mm]
The key notion in this paper is that of a modulus of uniform 
Fej\'er monotonicity i.e. a bound $\exists k\le \chi(r,n,m)$ 
for the following uniform strengthening of `Fej\'er monotone'
\[ \forall r,n,m\in\N\exists k\in\N\forall p\in X\left( 
p\in AF_k\to\forall l\le m \left( d(x_{n+l},p)<d(x_n,p)+\frac{1}{r+1}
\right)\right). \]
If $X$ is compact and $F$ satisfies an appropriate closedness condition 
w.r.t. the sets $AF_k$, 
then `Fej\'er monotone' and `uniform Fej\'er monotone' are equivalent. 
However, moduli $\chi$ for uniform Fej\'er monotonicity can be extracted 
(based on results from logic) 
also in the absence of compactness, provided that the proof of the Fej\'er 
monotonicity 
is formalizable in a suitable context, and we provide such moduli 
$\chi$ in all our applications.  
\\[1mm] 
If $X$ is compact, $F$ satisfies some explicit closedness condition w.r.t. 
$AF_k$  
(Definition \ref{explicit-closed}) and $(x_n)$ (in addition to being Fej\'er 
monotone) possesses approximate $F$-points, i.e. 
\[ (**)\quad \forall k\in\N\,\exists n\in\N \,(x_n\in AF_k), \] 
then $(x_n)$ converges to a point in $F$ 
(see Proposition \ref{general-GH-fejer} and the 
remark thereafter). 
\\[2mm]
The main 
general quantitative theorem in our paper (Theorem \ref{main-step-general-GH-fejer}) 
transforms (given $k,g$) any 
modulus of total boundedness $\gamma$ (a quantitative way to express the 
total boundedness of $X$, see Section \ref{logic-section}), 
any bound $\Phi$ on $(**)$ and any modulus $\chi$ 
of uniform Fej\'er monotonicity into a rate $\Psi(k,g,\Phi,\chi,\gamma)$ 
of metastability $(*)$ of $(x_n).$ If, moreover, $F$ is uniformly closed 
w.r.t. $AF_k$ (Definition 
\ref{uniformly-closed}), 
which e.g. is the case when $F$ and $AF_k$ are, respectively, the fixed point and the $1/(k+1)$-approximate fixed point set of a uniformly continuous mapping $T$, 
then one can also arrange that all the points in the interval of 
metastability $[N,N+g(N)]$ belong to $AF_k$ (Theorem  
\ref{finitization-general-GH-Fejer}). \\[1mm] 
$\Psi$ is the $P$-times iterate of 
a slightly massaged (with $\chi,\Phi$) version of $g,$ where $P$ only depends 
on $\gamma,k$ (but not on $g$). In particular, this yields that a rate 
of convergence for $(x_n)$ (while not being 
computable) is effectively learnable 
with at most $P$-many mind changes and a learning strategy which - essentially 
- is $\Phi\circ\chi$ (see \cite{Kohlenbach-Safarik} for more on this). That 
a primitive recursive iteration of $g$ is unavoidable follows from the 
fact that even for most simple cases of Fej\'er monotone fixed point 
iterations $(x_n)$ in $[0,1]$ 
the Cauchy property of $(x_n)$ implies the Cauchy property of monotone 
sequences in $[0,1]$ (see \cite{Neumann}) 
which is equivalent to $\Sigma^0_1$-induction 
(\cite{Kohlenbach(PRA)}(Corollary 5.3)). See also the example at the 
end of section \ref{section-quantitative}. 
\\   
A variant of Theorem \ref{finitization-general-GH-Fejer} holds 
even without any closedness assumption, if $(x_n)$ 
not only possesses $AF_k$-points for every $k$ but is 
asymptotic regular 
\[ \forall k\in\N\,\exists n\in\N\,\forall m\ge n\ \left( 
x_m\in AF_k\right)\] with a rate of metastability 
$\Phi^+$ for this property instead of the approximate $F$-point bound 
$\Phi$ (Theorem \ref{Cauchy+as-reg-metastability}). \\[1mm] 
In all these results we actually permit a more general form of 
Fej\'er-monotonicity, where instead of $(+)$ one has 
\[ (++)\quad H(d(x_{n+m},p))\le G(d(x_n,p)), \ \mbox{for all $n,m\in\N$ and $p\in F$}. \]
and $G,H:\R_+\to\R_+$ are subject to very general conditions (this e.g. 
is used in the application to asymptotically nonexpansive mappings). 
\\[1mm] 
As is typical for such quantitative `finitizations' of noneffective 
convergence results, it is easy to incorporate a summable sequence 
$(\varepsilon_n)$ of error terms in all the aforementioned results 
which covers the important concept of `quasi-Fej\'er-monotonicity' due to 
\cite{Ermolev}
(see Section \ref{section-quasi}). 
As a consequence of this, one can also incorporate 
such error terms in the iterations we are considering in this paper. 
However, for the sake of better readability we will not carry this out 
in this paper (but see \cite{Kohlenbach2015} for an application of this to  
convex feasibility problems in CAT($\kappa$)-spaces).  
\\[1mm] The results mentioned so far hold for arbitrary sets $F=\bigcap_kAF_k$ 
provided that we have the various moduli as indicated. In the case where 
$AF_k$ can be written as a purely universal formula and we have - 
sandwiched in between $AF_{k+1}\subseteq \tilde{AF}_{k}\subseteq AF_k$ - the sets 
$\tilde{AF}_k$ which are given by a purely existential formula 
(which is the case for $AF_k=\{ p\in X \mid d(p,Tp)\le 1/(k+1)\}$ with 
$\tilde{AF}_k=\{ p\in X \mid  d(p,Tp)<1/(k+1)\}$), then the logical 
metatheorems from \cite{Koh05a,GerKoh08,Kohlenbach(book)} {\bf guarantee} 
the extractability of explicit and highly uniform moduli $\chi$ from 
proofs of (generalized) Fej\'er monotonicity, as well as approximate 
fixed point bounds or metastability rates for asymptotic regularity 
from proofs of the corresponding properties if these proofs can be carried 
out in suitable formal systems as in all our applications. 
\\[1mm] The paper is organized as follows: in Section \ref{logic-section} 
we discuss the background from mathematical logic, i.e. so-called logical 
metatheorems (due to the first author in \cite{Koh05a}, see also 
\cite{GerKoh08,Kohlenbach(book)}) 
which provide tools for the extraction of highly uniform 
bounds from prima facie noneffective proofs of $\forall\exists$-theorems 
(which covers the case of metastability statements). Since our present paper
uses the context of totally bounded metric spaces we discuss this case 
in particular detail. Applying proof mining to a concrete proof 
results again in an ordinary proof in analysis and so one can read the proofs 
in this paper without any knowledge of logic which, however, was used 
by the authors to find these proofs. In Sections \ref{section-approximate} and 
\ref{section-Fejer} we develop the basic definitions and facts about 
the sets $F,AF_k,$ the notions of explicit and uniform closedness as 
well as (uniform) generalized $(G,H)$-Fej\'er monotone sequences. In Section 
\ref{section-quantitative} we establish our main general quantitative theorems 
which then will be specialized in our various applications. Section 
\ref{section-quasi} generalizes these results to the case of (uniform) 
quasi-Fej\'er monotone sequences. In Section \ref{section-F-FT} we interpret 
our results in the case where $F$ is the fixed point set of a selfmap 
$T$ (mostly of some convex subset of $X$) and provide numerous applications 
as mentioned above: in each of these cases we provide appropriate moduli 
of uniform (generalized) Fej\'er monotonicity $\chi$ and approximate 
fixed bounds bounds $\Phi$ (usually even rates of asymptotic regularity or 
metastable versions thereof) so that our general quantitative theorems can 
be applied resulting in explicit rates of metastability for $(x_n).$ In the 
case of CAT(0)-spaces (resp. Hilbert spaces), these $\Phi$'s become 
quadratic in the error $1/(k+1).$ In 
Section \ref{maximal-monotone} we do the same for the case where $F$ is 
the set of zeros of a maximal monotone operator and provide the 
corresponding moduli for the proximal point algorithm. \\[1mm] The 
results in this paper are based on compactness arguments. Without compactness 
one in general has only weak convergence for Fej\'er monotone sequences 
but in important cases weakly 
convergent iterations can be modified to yield strong convergence even 
in the absence of compactness (see e.g. \cite{BauCom10}). 
This phenomenon is known 
from fixed point theory where Halpern-type variants of the weakly convergent 
Mann iteration yield strong convergence 
(\cite{Browder(67),Halpern(67),Wittmann(92)}). Even when only weak convergence 
holds one can apply the logical machinery to extract rates of metastability 
for the weak Cauchy property (see e.g. \cite{Kohlenbach(Baillon)} 
where this is done in the case 
of Baillon's nonlinear ergodic theorem). However, the bounds will be 
extremely complex. If, however, weak convergence is used only as 
an intermediate step towards strong convergence, one can often avoid the 
passage through weak convergence altogether and obtain much simpler rates 
of metastability (see e.g. \cite{Kohlenbach(Browder)} where this has been carried out in 
particular for Browder's classical strong convergence theorem of 
the resolvent of a nonexpansive operator in Hilbert spaces, as well as 
\cite{Koernlein/Kohlenbach14}). We believe 
that it is an interesting future research project to adapt these techniques 
to the context of Fej\'er monotone sequences. 
\\[2mm] {\bf Notations:} $\N$ and $\N^*$ denote the set of natural numbers 
including $0$ resp. without $0$ and $\R_+$ are the nonnegative reals.

\section{Quantitative forms of compactness}\label{logic-section}

Let $(X,d)$ be a metric space. We denote with $B(x,r)$ (resp. $\ol{B}(x,r)$)  the open (resp. closed) ball with center $x\in X$ and radius $r>0$. 

Let us recall that  a nonempty subset $A\se X$ is {\em totally bounded} if for every $\eps>0$ there exists an $\eps$-net of $A$, i.e. there are $n\in\N$ and 
$a_0,a_1\ldots, a_n\in X$ such that $A\se \bigcup_{i=0}^nB(a_i,\eps)$. This is equivalent with the existence of a $1/(k+1)$-net for every $k\in\N$.

\begin{definition}
Let $\emptyset \ne A\se X$. We call $\alpha:\N\to\N$ a {\em I-modulus of total boundedness} for 
$A$ if for every $k\in\N$ there exist elements $a_0,a_1,\ldots, a_{\alpha(k)}\in X$ such that 
\beq
\forall x\in A\,\exists \,0\le i\le \alpha(k)\,\left(d(x,a_i)\le \frac1{k+1} \right). \label{def-modulus-tot-I}
\eeq
\end{definition}
Thus, $A$ is totally bounded iff $A$ has a I-modulus of total boundedness. In this case, we also say that 
$A$ is totally bounded with I-modulus $\alpha$. One can easily see that any totally bounded set is bounded: 
given a I-modulus $\alpha$ and $a_0, \ldots, a_{\alpha(0)}\in X$ such that \eqref{def-modulus-tot-I} is satisfied for $k=0$, 
$b:=2+\max\{d(a_i,a_j)\mid 0\le i,j\le \alpha(0)\}$ is an upper bound on the diameter of $A$.\\[1mm] 
We now give an alternative characterization of total boundedness used in the context of proof mining first in \cite{Ger08}:

\begin{definition}
Let $\emptyset \ne A\se X$.  We call $\gamma:\N\to\N$ a {\em II-modulus of total boundedness} for 
$A$ if for any $k\in\N$ and for any sequence $(x_n)$ in $A$  
\beq
\exists \,0\leq i<j\le \gamma(k)\,\left(d(x_i,x_j)\le \frac1{k+1}\right). \label{def-modulus-tot-II}
\eeq
\end{definition}

\begin{remark} The logarithm of the smallest possible value for a I-modulus 
of total boundedness is also called the $1/(k+1)$-entropy of $A$ while the 
logarithm of the optimal II-modulus is called the $1/(k+1)$-capacity 
of $A$ (see e.g. \cite{Lorentz}).
\end{remark}

\begin{proposition}\label{prop-mtb}
Let $\emptyset \ne A\se X$.
\be
\item\label{mtb-I-II} If  $\alpha$ is a I-modulus of total boundedness for $A$, then 
$\gamma(k):=\alpha(2k+1)+1$ is a II-modulus of total boundedness for $A$.
\item\label{mtb-II-I} If $\gamma$ is a II-modulus of total boundedness for $A$, then $\alpha(k):=\gamma(k)-1$ 
is a I-modulus of total boundedness (so, in particular, $A$ is totally bounded).  
\ee
\end{proposition}
\begin{proof}
\be
\item   Let $a_0, \ldots, a_{\alpha(2k+1)}\in X$ be such that \eqref{def-modulus-tot-I} is satisfied, hence  for all $x\in A$ there exists $0\le i\le \alpha(2k+1)$ such that
$\ds d(x,a_i)\le \frac1{2k+2}$. Applying the pigeonhole principle to $x_0,x_1,\ldots, x_{\alpha(2k+1)+1}$, we get $0\le i<j\le \alpha(2k+1)+1$, 
 such that $x_i$ and $x_j$ are in a ball of radius 
$\ds  \frac1{2k+2}$ around the same $a_l$ with $0\le l\le \alpha(2k+1)$. It follows that $\ds d(x_i,x_j)\le \frac{1}{k+1}$, hence \eqref{def-modulus-tot-II} holds.
\item First, let us remark that $\gamma(k)\geq 1$ for all $k$, hence $\alpha$ is well-defined. 
Assume by contradiction that $\alpha(k):=\gamma(k)-1$ is not a I-modulus of total boundedness, i.e. there exists $k\in\N$ such that
\[ (*) \quad \forall a_0,\ldots,a_{\gamma(k)-1}\in X 
\,\exists x\in A\,\forall \,  0\le i\le \gamma(k)-1 \, \left(d(x,a_i)>
\frac{1}{k+1}\right). \]
By induction on $l\le \gamma(k)$ we show that 
\[ (**)\quad \exists \beta_0,\ldots,\beta_l\in A \,\forall \, 0\le i<j\le l \ 
\left(d(\beta_i,\beta_j)> \frac{1}{k+1}\right),\] 
which, for $l:=\gamma(k)$, contradicts the assumption that 
$\gamma$ is a II-modulus of total boundedness. \\ 
$l=0$: Choose $\beta_0\in A$ arbitrary. \\ 
$l\mapsto l+1\le \gamma(k):$ Let $\beta_0,\ldots,\beta_l$ be as in 
$(**).$ By $(*)$ applied to 
\[a_i:=\begin{cases} \beta_i & \text{ if } i\le l\\
        \beta_l & \text{ if } l<i\le \gamma(k)-1
       \end{cases}
\]
we get $x\in A$ such that $\ds d(x,\beta_i) > \frac{1}{k+1}$  for all $i\le l$. Then $\beta_0,\ldots,\beta_l,\beta_{l+1}:=x$ satisfies 
$(**).$
\ee
\end{proof}

Note that the existence of a $1/(k+1)$-net in the proof of Proposition \ref{prop-mtb}.\eqref{mtb-II-I} is noneffective. In particular, there is no 
effective way to compute a bound on $A$ from a II-modulus of total boundedness. This seemingly disadvantage actually will allow us to 
extract bounds of greater uniformity from proofs of statements which do not explicitly refer to such a bound (see below).

\subsection{General logical metatheorems for totally bounded metric 
spaces}\label{logical-meta-tb-metric}

In \cite{Koh05a}, the first author introduced so-called logical metatheorems 
for bounded metric structures (as well as for normed spaces and other 
classes of spaces).\footnote{In this discussion we focus on the case 
of metric spaces.} Here systems 
${\cal T}^{\omega}$ of arithmetic and analysis in the language of functionals 
of all finite types are extended by an abstract metric space $X$ whose metric is supposed to be bounded by $b\in\NN$ resulting 
in a system ${\cal T}^{\omega}[X,d]$. 
Consider now a  ${\cal T}^{\omega}[X,d]$-proof 
of a theorem of the following 
form, where $P$ is some concrete complete separable metric space and 
$K$ a concrete compact metric space:\footnote{For simplicity, we only consider here 
some special case. For results in full generality see \cite{Koh05a,GerKoh08,
Kohlenbach(book)}.}
 \[(+) \quad \left\{ \ba{l} \forall u\in P\,\forall v\in K\,\forall x\in X\, \forall 
 y\in X^{\NN}\, 
\forall\, T:X\to X\ \\[1mm] \hspace*{1cm} 
(A_{\forall} (u,v,x,y,T)\to \exists n\in\NN\,
B_{\exists}(u,v,x,y,T)),\ea \right.\]
where $A_{\forall},B_{\exists}$ are purely universal resp. purely existential 
sentences (with some restrictions on the types of the quantified variables). 
Then from the proof one can extract (using a method from proof theory
called monotone functional interpretation,  due to the first author, see 
\cite{Kohlenbach(book)} for details on all this)
a computable uniform bound `$\exists n\le 
\Phi(f_u,b)$' on `$\exists n\in
\NN$' which only depends on some representation $f_u$ of $u$ in $P$ and a bound $b$ 
of the metric. In particular, $\Phi$ does not depend on $v,x,y,T$ 
nor on the space $X$ 
(except for the bound $b$). In most of 
our applications $P$ will be $\NN$ or $\NN^{\NN}$ (with the 
discrete and the Baire metric, respectively) in which case $u=f_u.$ 
In the cases $\RR$ or $C[0,1]$, however, $f_u$ is some concrete fast 
Cauchy sequence (say of Cauchy rate $2^{-n}$) of rationals representing $u\in \RR$ resp. a pair 
$(f,\omega)$ with $f\in C[0,1]$ and some modulus of uniform continuity 
$\omega$ for $f$ in the case of $C[0,1].$ $f_u$ can always be encoded 
into an element of $\NN^{\NN}.$ \\[1mm]
$\Phi$  
has some restricted subrecursive complexity 
which reflects the strength of the mathematical axioms from 
${\cal T}^{\omega}$ used in the proof. 
In most applications, $\Phi$ is at most of so-called primitive recursive complexity.\\[1mm]
As discussed in \cite{GerKoh08} and \cite[Application 18.16, p. 464]{Kohlenbach(book)}, the formalization of 
the total boundedness of $X$ via the existence of a I-modulus of total 
boundedness $\alpha$ can be incorporated in this setting as follows: 
in order to simplify the logical structure of the axiom to be 
added it is convenient to combine all the individual $\varepsilon$-nets 
$a_0,\ldots,a_{\alpha(\varepsilon)}$ into one single sequence $(a_n)$ 
of elements in $X$ and to replace the quantification over $\varepsilon >0$ 
by quantification over $\NN$ via $ \varepsilon:=1/(n+1):$ 
\\[1mm] The theory $\cT^{\omega}[X,d,TOTI]$ of totally bounded metric spaces 
is obtained by adding to $\cT^{\omega}[X,d]$
\be
\item two constants $\alpha^{\N\to\N}$ and $a^{\N\to X}$ denoting a function
$\NN\to\NN$ and a sequence $\NN\to X,$ respectively, as well as \and 
\item one universal axiom:\footnote{The bounded number quantifier can be 
easily eliminated by bounded collection.}
\[
(TOT \,I)\quad \forall k^\N\forall x^X\exists N\le_\N \alpha(k)\,\left(d_X(x,a_N)\leq_\R \frac{1}{k+1}\right).
\]
\ee
It is obvious that $(TOT \,I)$ implies that $\alpha$ is a I-modulus of 
total boundedness of $X$ as defined before. 
Conversely, suppose $\alpha$ is 
such a modulus. Then $\alpha'(n):=\sum^n_{i=0}(\alpha(i)+1)$ 
satisfies $(TOT \,I)$ for the sequence $(a_n)$ obtained as the concatenation of the 
$1/(k+1)$-nets $a_0^k,\ldots,a_{\alpha(k)}^k$, $k=0,1,\ldots$.\\[1mm]
\details{Since $\alpha$ is a I-modulus, for every $n\in\N$ there exists a $1/(n+1)$-net $a_0^n,\ldots, a_{\alpha(n)}^n$. 
Let $a:\N\to X$ be defined by the concatenation of the finite sequences $a_0^n,\ldots, a_{\alpha(n)}^n$, $n\in\N$. Let $k\in\N$ and $x\in X$. 
By \eqref{def-modulus-tot-I}, there exists $0\leq P\leq \alpha(k)$ such that $\ds d(x,a_P^k)\leq \frac1{k+1}$. We have that 
\[a_P^k=a\left(\sum_{i=0}^{k-1}(\alpha(i)+1)+P+1\right).\]
Let $N:=\sum_{i=0}^{k-1}(\alpha(i)+1)+P+1$. Then \
\[N\leq \sum_{i=0}^{k-1}(\alpha(i)+1) + (\alpha(k)+1)=\alpha'(k).\]
}
Since $(TOT \,I)$ is purely universal, 
its addition does not cause any problems and the only change caused by 
switching from $\cT^{\omega}[X,d]$ to $\cT^{\omega}[X,d,TOTI]$ is that the 
extracted bound $\Phi$ will additionally depend on $\alpha$ (see 
\cite{Kohlenbach(book)} for details).\\[1mm] 
In \cite{GerKoh08}, the results from \cite{Koh05a} are extended 
to the case of unbounded metric spaces. Then the bound 
$\Phi$ depends, instead of $b$, on majorizing data $x^*\gtrsim^p_X x, \,y^*\gtrsim^p_{\NN\to X}y,\, 
T^*\gtrsim^p_{X\to X} T$ for $x,y,T$ relative to 
some reference point $p\in X$ (which usually will be identified with 
$x$). More precisely, the $p$-majorizability relation $\gtrsim^p$ is defined (for the cases at hand 
which are special cases of a general inductive definition for all function 
types over $\NN,X$ interpreted here over the full set-theoretic type 
structure, see \cite{Kohlenbach(book)}) as follows: 
\[ \ba{l} n^*\gtrsim^p_{\NN} n:=n^*,n\in\NN\wedge n^*\ge n, \\[1mm] 
\alpha^* \gtrsim^p_{\NN\to\NN} \alpha:=\alpha^*,\alpha\in\NN^{\NN}\wedge 
\forall n^*,n (n^*\ge n\to \alpha^*(n^*)\ge \alpha^*(n),\alpha(n)), \\[1mm] 
x^*\gtrsim^p_X x:= x^*\in\NN, x\in X\wedge x^*\ge d(p,x), \\[1mm] 
y^*\gtrsim^p_{\NN\to X} y:=y^*\in \NN^{\NN}, y\in X^{\NN}\wedge 
\forall n^*,n\in\NN \, (n^*\ge n\to y^*(n^*)\ge d(p,y(n))), \\[1mm]  
T^*\gtrsim^p_{X\to X} T:=T^*\in\NN^{\NN}, T\in X^X \wedge \\ \hspace*{2cm}
\forall n\in\NN\,\forall x\in X (n\ge d(p,x) \to T^*(n)\ge d(p,T(x)). \ea \] 
Note that $\gtrsim^p_{\NN}$ and $\gtrsim^p_{\NN\to\NN}$ 
actually do not depend on $p$, hence we shall denote them simply $\gtrsim_{\NN}$ and $\gtrsim_{\NN\to\NN}$, respectively.
Whereas a majorant $y^*$ exists for any sequence $y$ in $X$, it is a genuine 
restriction on $T$ to posses a majorant $T^*$. However, for large 
classes of mappings $T$ one can construct $T^*$, e.g. this is the 
case when $T$ is 
Lipschitz continuous (in the case of geodesic spaces also uniform 
continuity suffices) but also in general whenever $T$ maps bounded 
sets to bounded sets. 
\details{ Assume that $T$ maps bounded sets to bounded sets. Define $T^*:\N\to\N$ as 
follows $T^*(n)=\sup\{d(p,T(x))\mid d(x,p)\leq n\}$. Then $T$ is well defined, since 
$T(B(p,n))$ is bounded, hence there is $M\in\R$ such that $M\geq d(p,T(x))$ for all $x\in B(a,n)$. }
\\[1mm]
It is instructive to see what happens if we take the context of 
{\bf unbounded} metric spaces, i.e. - using the terminology 
from \cite{GerKoh08,Kohlenbach(book)} - ${\cal T}^{\omega}[X,d]_{-b}$ 
and add constants $\alpha:\NN\to\NN$ and 
$(a_n):\NN\to X$ as before. Then we need to provide majorants $\alpha^*,
a^*$ for these 
objects, which in the case of $\alpha$ can be simply done by stipulating 
$\alpha^*(n):=\max\{ \alpha(i) \mid \le n\},$ whereas for $(a_n)$ this requires 
- as above - a function $a^*:\NN\to\NN$ such that $a^*\gtrsim^p_{\NN\to X} a$. 
Then the bound $\Phi$ extractable from proofs of theorems of 
the form considered above will additionally also depend on $\alpha^*$ 
(i.e. on $\alpha$) and $a^*$. From these data one can easily compute a 
bound $b$ on $X$ (e.g. we may take $b:=2+2a^*(\alpha^*(0))$)
\details{We have that 
$$a^*(\alpha^*(0))=a^*(\alpha(0))=\max\{d(x,a_k)\mid k\leq \alpha(0)\}.$$
Let $y,z\in X$. Applying (TOT I) with $k=0$ we get $0\leq N_y, N_z\leq 1$ such that 
$d(y,a_{N_y}), d(z,a_{N_z})\leq 1$. It follows that 
\bua
d(y,z)\leq d(y,a_{N_y})+d(z,a_{N_z})+d(a_{N_y},x)+d(x,a_{N_z})\leq 2+ 2a^*(\alpha^*(0)). 
\eua
}
and, conversely, given such a bound 
$b$ one can simply take $a^*(n):=b$. So, adding $(TOT \,I)$ gives in both 
contexts the same results w.r.t. the extractability of bounds $\Phi$ and 
their uniformity. 
This situation, however, changes if we consider the axiomatization 
based on the II-modulus of total boundedness in the setting of unbounded 
metric structures: \\[1mm] 
The theory $\cT^{\omega}[X,d,TOTII]_{-b}$ 
of totally bounded metric spaces is obtained by adding to $\cT^{\omega}[X,d]_{-b}$
\be
\item one constant $\gamma^{\N\to\N}$  and 
\item one universal axiom: 
\[
(TOT\, II)\quad \forall k^\N\forall x^{\N\to X}\exists I,J\le_\N \gamma(k)\,\left(I<_\N J \,\wedge\, d_X(x_I,x_J)\leq_\R \frac1{k+1} \right).
\]
\ee
Due to the absence of the sequence $(a_n)$ from this axiomatization, the 
extracted bounds will only depend on $\gamma$ instead of $\alpha,a^*$ 
(or $\alpha,b$). This results in a strictly greater uniformity of the bounds 
as the following example shows.\\[1mm]
Consider the sequence $(X,d_n)$ of metric spaces defined as follows: 
$$X:=\{ 0,1\}, \ d_n(0,1):=d_n(1,0):=n, \ d_n(0,0)=d_n(1,1)=0.$$ 
It is easy to see that $\gamma(n):=2$ is a common II-modulus of total 
boundedness for all the spaces $(X,d_n)$ (since any sequence of 3 
elements of $X$ has to repeat some element), 
while the diameter of $(X,d_n)$ tends 
to infinity as $n$ does. Hence our bounds $\Phi$ will be uniform for all 
the spaces $(X,d_n)$ which first might look impossible since, 
after all, $(TOT\, II)$ {\bf does} imply that $X$ is bounded, i.e. 
\[ (++)\quad \exists b\in\NN\,\forall x,y\in X \, (d(x,y)< b).\]
However, 
$(++)$ is of the form $\exists\forall$, which is not allowed in statements 
of the form $(+)$ considered above. Noneffectively, $(++)$ can be 
equivalently reformulated as 
\[ (++)' \quad \forall (x_n),(y_n)\in X^{\NN} \,\exists N\in\NN \, (d(x_N,y_N)< N),\] 
which {\bf is} of the form $(+)$, so that the aforementioned uniform 
bound extraction applies (given majorants $x^*,y^*$ for $(x_n),(y_n)$). Indeed, 
define recursively 
\[ n_0:=0, \quad n_{k+1}:=\left\lceil\max_{i,j\le k} \{ n_k,d(x_{n_i},y_{n_j}),
d(x_{n_i},x_{n_j}),d(y_{n_i},y_{n_j})\} +3\right\rceil, \] 
which can easily be effectively bounded using only $d$ and $x^*,y^*$.

\begin{proposition}
For any metric space $X$ with II-modulus of total boundedness $\gamma$ we have: 
\[ \exists N\le n_{\gamma(0)} \,(d(x_N,y_N)< N). \] 
\end{proposition}
\begin{proof}
Suppose that $\forall k\le \gamma(0) (d(x_{n_k},y_{n_k})\geq n_k).$ 
Then, for all $k\le\gamma(0),$ one of the two cases 
\[ (1)  \ \forall i< k \,(d(x_{n_{k}},x_{n_i}),d(x_{n_{k}},y_{n_i})>1)\]
or 
\[ (2) \ \forall i< k \,(d(y_{n_{k}},x_{n_i}),d(y_{n_{k}},y_{n_i})>1)\]
holds since, otherwise, 
$d(x_{n_{k}},y_{n_{k}})\le n_{k-1}+2< n_{k}$.
\details{Assume that (1) and (2) do not hold for some $k$. Then there exists $I<k$ such that $d(x_{n_{k}},x_{n_I})\leq 1$ or $d(x_{n_{k}},y_{n_I}) \leq 1$ and there 
exists $J<k$ such that $d(y_{n_{k}},x_{n_J})\leq 1$ or $d(y_{n_{k}},y_{n_J})\leq 1$. We have four cases, but it is enough to consider the following two:
\be
\item $d(x_{n_{k}},x_{n_I})\leq 1$ and $d(y_{n_{k}},x_{n_J})\leq 1$. It follows that
\bua
d(x_{n_{k}},y_{n_{k}}) &\leq & d(x_{n_{k}},x_{n_I})+d(y_{n_{k}},x_{n_J})+d(x_{n_I},x_{n_J})\leq 2+n_{k-1}\leq n_k.
\eua
\item $d(x_{n_{k}},x_{n_I})\leq 1$  and $d(y_{n_{k}},y_{n_J})\leq 1$.
It follows that
\bua
d(x_{n_{k}},y_{n_{k}}) &\leq & d(x_{n_{k}},x_{n_I})+d(y_{n_{k}},y_{n_J})+d(x_{n_I},y_{n_J})\leq 2+n_{k-1}\leq n_k.
\eua
\ee
} 
Define a sequence 
$z_0,\ldots,z_{\gamma(0)}$ as follows: for $k\le \gamma(0)$ put 
$z_k:=x_{n_k}$, if $(1)$ holds, and $z_k:=y_{n_k}$, otherwise (which 
implies that $(2)$ holds). Then $d(z_i,z_j)>1$ whenever $0\le i<j\le 
\gamma(0)$ which, however, contradicts the definition of $\gamma.$ 
Hence $\exists k\le\gamma(0) \ (d(x_{n_k},y_{n_k}) < n_k).$ Since 
$n_k\le n_{\gamma(0)},$ the claim follows.
\end{proof}
\begin{remark}
As mentioned already, logical metatheorems of the form discussed above have 
also been established for more enriched structures such as W-hyperbolic 
spaces, uniformly convex W-hyperbolic spaces, $\R$-trees, $\delta$-hyperbolic 
spaces (in the sense of Gromov) and CAT(0)-spaces as well 
normed spaces, uniformly convex normed spaces, complete versions of 
these spaces and Hilbert spaces. Most recently, also abstract $L^p$- and 
$C(K)$-spaces have been covered (\cite{GueKoh14}). 
In the normed case, the reference point 
$p\in X$ used in the majorization relation will always be the zero vector 
$0_X$ (see \cite{GerKoh08,Koh05a,Kohlenbach(book),Leu06,Leu07,GueKoh14} 
for all this). In all these cases one can add the requirement of $X$ (or 
of some bounded subset in the normed case) to be totally bounded with 
moduli of total boundedness in the form I or II as above. Thus the 
applications given in this paper can be viewed as instances of corresponding 
logical metatheorems. 
\end{remark}

\subsection{Examples}
In this subsection we give simple examples of II-moduli of total boundedness that are computed explicitly. Although some of the proofs are straightforward we include them for completeness.

\begin{example}\label{ex-tb-01}
Let $A=[0,1]$ be the unit interval in $\R$. Then $\gamma:\N \to \N$, $\gamma(k)= k+1$ is a II-modulus of total boundedness for $A$.
\end{example}
\begin{proof}
Let $k \in \N$ and $(x_n)$ be a sequence in $A$. Divide the interval $[0,1]$ into $k+1$ subintervals of equal length $1/(k+1)$. Applying the pigeonhole principle we obtain that there exist $0 \le i < j \le k+1$ such that $|x_i - x_j|\le 1/(k+1)$.
\end{proof}

\begin{example}\label{ex-bd-set-Rn}
Let $A$ be a bounded subset of $\mathbb{R}^n$ and $b > 0$ be such that $\|a\|_2 \le b$ for every $a \in A$. Then $\gamma:\N \to \N$, $\ds\gamma(k) = \left\lceil2(k+1)\sqrt{n}b\right\rceil^n$ is a II-modulus of total boundedness for $A$.
\end{example}
\begin{proof}
Let $k \in \N$ and $(x_p) \subseteq A$. Denote $N = \left\lceil2(k+1)\sqrt{n}b\right\rceil$. Clearly, $A$ is included in the cube $[-b,b]^n$. Divide this cube into $N^n$ subcubes of equal side lengths $2b/N$. The diameter of each subcube is $ 2b\sqrt{n}/N \le 1/(k+1)$. Applying the pigeonhole principle we obtain that there exist $0 \le i < j \le N^n$ such that $\|x_i - x_j\|_2\le 1/(k+1)$.
\end{proof}

\begin{example} \label{ex-closure}
Let $(X,d)$ be a metric space and $A \subseteq X$ totally bounded with II-modulus of total boundedness $\gamma$. Then the closure of $A$ is totally bounded with II-modulus of total boundedness $\gamma$.
\end{example}
\begin{proof}
Let $k \in \N$ and $(x_n) \subseteq \ol{A}$. Take $m \in \N$. Then there exists a sequence $(a_n) \subseteq A$ such that $d(x_n,a_n) \le 1/(m+1)$. Since $\gamma$ is a II-modulus of total boundedness for $A$, there exist $0 \le i < j \le \gamma(k)$ such that $d(a_i,a_j) \le 1/(k+1)$. Thus,
\[d(x_i,x_j) \le d(x_i,a_i) + d(a_i,a_j) + d(a_j, x_j) \le \frac{1}{k+1} + \frac{2}{m+1}.\]
Hence, there exist $0 \le i < j \le \gamma(k)$ and $(m_s)$ a strictly increasing sequence of natural numbers such that for every $s \ge 0$, $d(x_i,x_j) \le 1/(k+1) + 2/(m_s+1)$, from where $d(x_i,x_j) \le 1/(k+1)$.
\end{proof}

\begin{example} \label{ex-conv-hull}
Let $(X, \| \cdot \|)$ be a normed space and $A \subseteq X$ totally bounded with II-modulus of total boundedness $\gamma$. Then the convex hull 
$\text{co}(A)$ of $A$ is totally bounded with II-modulus of total boundedness 
\[\ol{\gamma}(k) = \left\lceil\ 2(m+1)\sqrt{n+1}\ \right\rceil^{n+1},\]
where $n = \gamma(4k+3) - 1$, $m = \left\lceil 2(k+1)(n+1)\left(b + 1/(4k+4)\right)\right\rceil -1$ and $b > 0$ is such that $\|a\| \le b$ for all $a \in A$.
\end{example}
\begin{proof}
Let $k \in \N$ and $(y_p) \subseteq \text{co}(A)$. Denote $r_k=1/(4k+4)$. By Proposition \ref{prop-mtb}.(\ref{mtb-II-I}), there exist $a_0,\ldots, a_n \in A$ such that 
\[A \subseteq \bigcup_{l=0}^{n}\ol{B}\left(a_l,r_k\right).\] 
Let $p \in \N$. Then there exist $s(p) \in \N$ and for $l = 0, \ldots, s(p)$, $t_l^p \in [0,1]$ and $x_l^p \in A$ such that $\ds \sum_{l=0}^{s(p)} t_l^p = 1$ and $\ds y_p = \sum_{l=0}^{s(p)}t_l^p x_l^p$. We can assume that $s(p)=n$ and $x_l^p \in \ol{B}\left(a_l,r_k\right)$ for $l = 0, \ldots, n$. This can be done because we can group any two points that belong to the same ball in the following way: suppose $x_0^p,x_1^p \in \ol{B}\left(a_0,r_k\right)$. Denote 
\[\ol{x}_0^p = \frac{t_0^p}{t_0^p + t_1^p}x_0^p + \frac{t_1^p}{t_0^p+t_1^p}x_1^p \in \ol{B}\left(a_0,r_k\right).\] 
Then, $y_p = (t_0^p + t_1^p)\ol{x}_0^p + t_2^p x_2^p + \ldots + t_n^p x_n^p$. Note that if in this way we obtain less than $n+1$ points in the convex combination then we add the corresponding $a_l$'s multiplied by $0$.

For $p \in \N$, $t^p = (t_0^p, \ldots, t_n^p) \in \R^{n+1}$ and $\|t^p\|_2 \le 1$. By Example \ref{ex-bd-set-Rn}, there exist $0 \le i < j \le \left\lceil\ 2(m+1)\sqrt{n+1}\ \right\rceil^{n+1}$ such that 
\[\|t^i - t^j\|_2 \le \frac{1}{m+1} \le \frac{2r_k}{\left(b+r_k\right)(n+1)}.\]
Then,
\bua
\|y_i - y_j\| & = & \left\|\sum_{l=0}^n (t_l^i x_l^i - t_l^j x_l^j )\right\| \le \left\|\sum_{l=0}^n (t_l^i x_l^i - t_l^i x_l^j )\right\|  + \left\|\sum_{l=0}^n (t_l^i x_l^j - t_l^j x_l^j )\right\| \\
& \le &\sum_{l=0}^n t_l^i \|x_l^i-x_l^j\|  + \|x_l^j\|\sum_{l=0}^n |t_l^i - t_l^j|\\
& \le & 2r_k\sum_{l=0}^n t_l^i + \left(b+r_k\right)(n+1)\|t^i-t^j\|_2 \le 2r_k + 2r_k = \frac{1}{k+1}.
\eua
\end{proof}

\section{Approximate points and explicit closedness}
\label{section-approximate}

In the following, $(X,d)$ is a metric space and $F\subseteq X$ a nonempty subset.  We assume that  
\[
F=\bigcap_{k\in\N}\tilde{F}_k,
\]
where $\tilde{F}_k\se X$ for every $k\in\N$ and we say that the family  $(\tilde{F}_k)$ is a {\em representation} of $F$. 
Of course, $F$ has a trivial representation, by letting $\tilde{F}_k:=F$ for all $k$. 
Naturally, we think of more interesting choices for $\tilde{F}_k$, as we look at 
\[
AF_k:=\bigcap_{l\le k}\tilde{F}_l
\]
as some weakened approximate form of $F$. A point $p\in AF_k$  is said to be a
{\em $k$-approximate $F$-point}. \\[1mm]
In the following we always view $F$ not just as a set but we suppose it is equipped 
with a representation $(\tilde{F}_k)$ to which we refer implicitly in many 
of the notations introduced below.  

Let $(x_n)$ be a sequence in $X$.

\begin{definition}
We say that  
\be
\item $(x_n)$ has approximate $F$-points if  $\forall k\in\NN \,\exists N\in \NN\,\,  (x_N\in AF_k)$.
\item $(x_n)$ has the liminf property w.r.t. $F$ if $\forall k,n\in \NN\,\exists N\in\N \,\,\, (N\ge n \text{~and~} x_N\in AF_k)$.
\item $(x_n)$ is asymptotically regular w.r.t. $F$ if $\forall k\in \NN\,\exists N\in\NN\,\forall m\ge N \,\,\, (x_m\in AF_k)$.
\ee
\end{definition}

\begin{lemma}\label{xkafk--subseq-liminf}
Assume that $x_k\in AF_k$ for all $k\in\N$. Then any subsequence of $(x_n)$ has the liminf property w.r.t. $F$.
\end{lemma}
\begin{proof}
Let $(x_{m_l})$ be a subsequence of $(x_n)$. Then $m_l\ge l$ and $x_{m_l}\in AF_{m_l}$ for all $l\in\N$. 
Let $k,n\in\N$ and take $N\geq \max\{n,k\}$. Then $x_{m_N}\in  AF_{m_N}\se AF_N\se AF_k$. 
\end{proof}

\begin{definition} \label{explicit-closed}
We say that  $F$ is {\em explicitly closed} (w.r.t. the representation $(\tilde{F}_k)$) if 
\[
\forall p\in X \,\left( \forall 
N,M\in\N (AF_M\cap \ol{B}\left(p, 1/(N+1)\right)\ne \emptyset)\ra p\in F\right).
\]
\end{definition}

One can easily see that if  $F$  is explicitly closed, then $F$ is closed. 
$F$ in particular is explicitly closed if all the sets $AF_k$ are (and so 
if all the sets $\tilde{F}_k$ are closed). Hence 
closedness of $F$ is  equivalent to 
explicit closedness of $F$ w.r.t. the trivial representation.
\details{
Let $p\in X\setminus F$ be arbitrary, so $p\not\in AF_k$ for some $k\in\N$. As a consequence,
there exist $M,N$ s.t. for all $q\in B_{\frac1{N+1}}(p)$ we have $q\not\in AF_M$. It follows that $B_{\frac1{N+1}}(p)\subset X\setminus F$. 
Hence  $X\setminus F$ is open. 

Furthermore, $F$ is explicitly closed w.r.t. the trivial representation

\begin{tabular}{ll}
 iff & $\forall p\in X\,\exists N\in\NN\,( F\cap \ol{B}_{\frac1{N+1}}(p)\ne \emptyset\ra p\in F)$\\
 iff & $\forall p\in X\setminus F\,\exists N\in\NN\,( F\cap \ol{B}_{\frac1{N+1}}(p)=\emptyset)$\\
 iff & $\forall p\in X\setminus F\,\exists N\in\NN\,( \ol{B}_{\frac1{N+1}}(p)\se X\setminus F)$\\
 iff & $\forall p\in X\setminus F\,\exists N\in\NN\,( B_{\frac1{N+1}}(p)\se X\setminus F)$\\
 iff & $X\setminus F$ is open.
\end{tabular}
}
The property of being explicitly closed can be re-written (pulling also the 
quantifier hidden in `$p\in F$' in front) in the following equivalent form
\[
\forall k\in\NN\,\forall p\in X\,\exists N,M\in\NN\,\left( AF_M\cap \ol{B}\left(p,1/(N+1)\right)\ne \emptyset\ra p\in AF_{k}\right).
\]
This suggests the following uniform strengthening of explicit closedness:

\begin{definition} \label{uniformly-closed}
$F$ is called {\em uniformly closed with moduli}
$\delta_F,\omega_F:\NN\to\NN$ if 
\[\forall k\in\NN\,\forall p,q\in X\, \,
\left( q\in AF_{\delta_F(k)} \text{~and~} d(p,q)\le \frac1{\omega_F(k)+1}  \rightarrow p\in AF_{k}\right).  \]
\end{definition}

\begin{lemma}\label{F-unifcont-liminf-F} 
Assume that $F$ is explicitly closed, $(x_n)$ has the liminf property w.r.t. $F$ and that $(x_n)$ converges strongly to $\widehat{x}$.  Then  $\widehat{x}\in F$.
\end{lemma}
\begin{proof}
Let $k\in\N$ be arbitrary.  Since $F$ is explicitly closed, there exist $M,N\in\NN$ such that 
\beq  \exists q\in X\,\left(
d(\widehat{x},q)\le \frac1{N+1} \text{~and~} q\in AF_{M}\right) \to \widehat{x}\in AF_k.
\label{pointwise} 
\eeq  
As $\ds\limn x_n=\widehat{x}$,  there exists $\tilde{N}\in\N$ such that  $d(x_n,\widehat{x})\leq \frac1{N+1}$ for all $n\geq \tilde{N}$. 
As $(x_n)$  has the liminf property w.r.t. $F$, we get that $x_K\in AF_{M}$ for some $K\ge \tilde{N}$.
Applying \eqref{pointwise} gives $\widehat{x}\in AF_k$. 
\end{proof}

\begin{lemma}\label{xn-cluster-point-F} 
Suppose that $X$ is compact, $F$ is explicitly closed and  that $(x_n)$ has approximate $F$-points. Then the set $\{ x_n| n\in\N\}$ has an adherent point  
$x\in F$. 
\end{lemma}
\begin{proof} We have for each $k\in\NN$ an $m_k\in\NN$ such that $x_{m_k}$ is a $k$-approximate $F$-point. 
Let $y_k:=x_{m_k}\in AF_k. $ Since $X$ is compact, the sequence $(y_k)$ has a convergent 
subsequence $(y_{k_n})$. Let $\ds x:=\limn y_{k_n}$. By Lemma \ref{xkafk--subseq-liminf}, $(y_{k_n})$ has the liminf property w.r.t. $F$. Apply now 
Lemma \ref{F-unifcont-liminf-F} to conclude that  $x\in F$.
\end{proof}

\section{Generalized Fej\'er monotone sequences} \label{section-Fejer}

In this section we give a generalization of Fej\'{e}r monotonicity, one of the most used methods for strong convergence proofs in convex optimization.

We consider functions $G:\RR_+\to\RR_+$ with the property 
\[ (G)\quad \mbox{If} \ a_n\stackrel{n\to\infty}{\to} 0,\ \mbox{then} \  G(a_n) \  
\stackrel{n\to\infty}{\to} 0 \] for all sequences $(a_n)$ in $\RR_+$. 

Obviously, $(G)$ is equivalent to the fact that there exists a mapping $\alpha_G:\N\to\N$ satisfying 
\beq 
\forall k\in\N \, \forall a\in\RR_+ \ 
\left( a\le \frac1{\alpha_G(k)+1} \to G(a) \le \frac1{k+1}\right). \label{def-alpha-G} 
\eeq
We say that such a mapping  $\alpha_G$ is a {\em $G$-modulus}. 

Any continuous $G:\RR_+\to\RR_+$ with $G(0)=0$ satisfies 
$(G)$ and any modulus of continuity of $G$ at $0$ is a $G$-modulus. 

We also consider functions $H:\RR_+ \to\RR_+$ with the property 
\[ (H)\quad \mbox{If} \  H(a_n) \  
\stackrel{n\to\infty}{\to} 0, \ \mbox{then} \  a_n\stackrel{n\to\infty}{\to} 
0\] for all sequences $(a_n)$ in $\RR_+$,  
which is kind of the converse of $(G)$.

Similarly, $(H)$ is equivalent to the existence of an {\em $H$-modulus} $\beta_H:\N \to\N$ such that
 \beq
 \forall k\in\N \, \forall a\in\RR_+ \ 
\left(H(a)\le \frac1{\beta_H(k)+1} \to a \le  \frac1{k+1}\right). \label{def-beta-H} 
\eeq

\mbox{}

Let $(x_n)$ be a sequence in the metric space $(X,d)$ and $\emptyset \ne F\se X$.

\begin{definition} 
$(x_n)$ is {\em $(G,H)$-Fej\'er monotone} w.r.t. $F$ if for all $n,m\in\N$ and all $p\in F$,
\[H(d(x_{n+m},p))\le G(d(x_n,p)).\] 
\end{definition}

Note that  the usual notion of being `Fej\'er monotone' is  just $(id_{\RR_+},id_{\RR_+})$-Fej\'er monotone. 

The following lemma collects some useful properties of generalized Fej\'er monotone sequences.

\begin{lemma}\label{GH-Fejer-basic}
Let $(x_n)$ be $(G,H)$-Fej\'er monotone w.r.t. $F$. 
\be
\item\label{GH-Fejer-basic-cluster} If $\{ x_n|n\in\N\}$ has an adherent 
point $\hat{x}\in F$, then $(x_n)$ converges to $\hat{x}$.
\item\label{GH-Fejer-basic-bounded} Assume that  $H$ has the property
\[(H1) \quad \mbox{If} \ H(a_n) \text{~is bounded},  \text{~then~} \  (a_n) \text{~is bounded} \]
for all sequences $(a_n)$ in $\RR_+$. Then $(x_n)$ is bounded.
\ee
\end{lemma}
\begin{proof}
\be
\item 
Let  $p\in\N$ be  arbitrary and  $K:=K_p$ be so that 
\[ d(x_{K},\hat{x})\le \frac1{\alpha_G(\beta_H(p)+1)+1} \] 
(such a $K$ has to exist by the assumption). 
Applying the fact that $(x_n)$ is $(G,H)$-Fej\'er monotone w.r.t. $F$ and \eqref{def-alpha-G},  we get that for all $l\in\NN$,
\[ H(d(x_{{K}+l},\hat{x}))\le 
G(d(x_{K},\hat{x}))  \le \frac1{\beta_H(p)+1}. \] 
Using now \eqref{def-beta-H},  it follows that $\ds d(x_{{K}+l},\hat{x}) \le \frac1{p+1}$ for all  $l\in\NN$.
 Hence $(x_n)$ converges to $\hat{x}$.
\item 
Since $(x_n)$ is $(G,H)$-Fej\'er monotone w.r.t. $F$ we have for $p\in F$ that $H(d(x_n,p)) \le G(d(x_0,p))$ for all $n\in\N$. Hence,  $(H(d(x_n,p)))$ 
is bounded and so, by $(H1)$,  $(d(x_n,p))$ is bounded.  
\ee
\end{proof}

As an immediate consequence of Lemma \ref{xn-cluster-point-F} and Lemma \ref{GH-Fejer-basic}.\eqref{GH-Fejer-basic-cluster}, we get 

\begin{proposition}\label{general-GH-fejer} 
Let $X$ be a compact metric space  and $F$ be explicitly closed. Assume that $(x_n)$ is  $(G,H)$-Fej\'er monotone with respect to $F$ 
and that $(x_n)$ has approximate $F$-points. Then $(x_n)$ converges to a point $x\in F$. 
\end{proposition}

\begin{remark} If in Proposition 
\ref{general-GH-fejer} we either weaken `compact' to `totally bounded' or drop the assumption on $F$ being explicitly closed, then the conclusion in general 
becomes false, in fact $(x_n)$ might not even be Cauchy (see Example \ref{general-fejer-FT-counter}).
\end{remark}

Let us recall that a metric space is said to be {\em boundedly compact} if every bounded sequence has a convergent subsequence.  One can easily see that $X$ is 
boundedly compact if and only if for every $a\in X$ and $r>0$ the closed ball $\ol{B}(a,r)$ is compact.

\begin{remark}\label{general-fejer-boundedly-compact}
The proof of Proposition \ref{general-GH-fejer} uses the compactness property only for the sequence $(x_n)$ and so it is enough to require that  $X$  is boundedly 
compact  and that the sequence at hand is bounded. As we prove above, this is the case if $H$ has the property $(H1)$ for all sequences $(a_n)$ in $\RR_+$.
\end{remark}

\subsection{Uniform  $(G,H)$-Fej\'er monotone sequences}

Being $(G,H)$-Fej\'er monotone w.r.t. $F$ can be logically re-written as 
\[\ba{ll}
\forall n,m\in\NN\,\forall p\in X\,\, \bigg(  & \forall k\in \NN (p\in AF_k) \ra  \\
& 
\ds \forall r\in\N\, \forall l\le m  \left(H(d(x_{n+l},p))< G(d(x_n,p)) +
\frac1{r+1} \right)\bigg),
\ea
\]
hence as
\[\forall r, n,m\,\forall p\,\exists k
\  \bigg( p\in AF_k \rightarrow \forall l\le m \left(H(d(x_{n+l},p))< G(d(x_n,p)) +
\frac1{r+1} \right)\bigg). 
\]

If $p\in AF_k$  can be written as a purely universal  formula (when formalized in the language of the 
systems used in the logical metatheorems from \cite{Koh05a,GerKoh08,Kohlenbach(book)}), then 
\[p\in AF_k \rightarrow  \forall l\le m \left(H(d(x_{n+l},p))< G(d(x_n,p)) +
\frac1{r+1} \right)\]
is (equivalent to) a purely existential formula. Hence one can use these metatheorems to extract 
a uniform bound on `$\exists k\in\NN$' (and so in fact a uniform realizer as  the formula is monotone in $k$)  which - e.g. for bounded $(X,d)$ - only 
depends on a bound on the metric and majorizing data of the other 
parameters involved but not on `$p$'. This motivates the next definition:

\begin{definition} 
We say that $(x_n)$ is {\em uniformly $(G,H)$-Fej\'er monotone} w.r.t. $F$ if for all $r, n,m\in\NN$, 
\[
\exists k\in\NN\, \forall p\in X\
\bigg( p\in AF_k \rightarrow \forall l\le m \left(H(d(x_{n+l},p))< G(d(x_n,p)) +
\frac1{r+1} \right)\bigg).
\]
Any upper bound (and hence realizer) $\chi(n,m,r)$ of `$\exists k\in \NN$'  is called a modulus of $(x_n)$ being (uniformly) 
$(G,H)$-Fej\'er monotone w.r.t. $F$.
\end{definition} 

If $G\!=\!H=\!id_{\RR_+}$, we say simply that $(x_n)$ is uniformly Fej\'er monotone w.r.t. $F$.

\begin{remark}\label{remark-unif-Fejer-Fejer}
\begin{enumerate} 
\item
A standard compactness argument shows that for $X$ compact, $F$ explicitly closed and 
$G,H$ continuous the notions `$(G,H)$-Fej\'er monotone w.r.t. $F$' and 
`uniformly $(G,H)$-Fej\'er monotone w.r.t. $F$' are equivalent.\details{??? TO WRITE???}
\item 
In Corollary \ref{qualitative-corollary} 
we will see, as a consequence of our quantitative metastable analysis 
of the proof of Proposition 
\ref{general-GH-fejer}, that the Cauchyness of $(x_n)$ holds 
even if we replace `compact' by `totally bounded' and drop the explicit closedness of $F$ {\bf provided} that we 
replace `$(G,H)$-Fej\'er monotone' by `uniform $(G,H)$-Fej\'er monotone'.
\item 
The equivalence between these notions can be proven (relative to 
the framework of ${\cal T}^{\omega}[X,d]$) for general bounded 
metric spaces $X$ and $F,G,H$ from the `nonstandard' uniform 
boundedness principle $\exists${\rm-UB}$^{X}$ studied in \cite{Kohlenbach(book)}.
Though being false for specific spaces $X,$ the use of $\exists${\rm-UB}$^{X}$ 
in proofs of statements of the form considered in our general bound-extraction 
theorems is allowed and the bounds extracted from proofs in 
${\cal T}^{\omega}[X,d]+\exists${\rm-UB}$^X$ 
will be correct in any 
bounded metric space $X$ (see \cite[Theorem 17.101]{Kohlenbach(book)}).
\end{enumerate}
\end{remark}

\section{Main quantitative results} \label{section-quantitative}

In this section, $(X,d)$ is a totally bounded metric space with a II-modulus of total boundedness $\gamma$ and $\emptyset\ne F\se X$. Furthermore, 
$G,H:\RR_+\to\RR_+$ satisfy $(G), (H)$ for all sequences  $(a_n)$ in $\RR_+$,  $\alpha_G$ is a $G$-modulus  and $\beta_H$ is an $H$-modulus.\\

Assume that $(x_n)$ has approximate $F$-points. We can then define the  mapping 
\beq
\varphi_F:\NN\to\NN, \quad \varphi_F(k)=\min \{m\in\N \mid x_m\in AF_k\} \label{def-phi-F}
\eeq
Thus, $x_{\varphi_F(k)}\in AF_k$ for all $k$ and $\varphi_F$ is monotone nondecreasing.  An {\em approximate $F$-point bound} for $(x_n)$ is any function
$\Phi:\NN\to\NN$ satisfying 
\beq
\forall k\in\NN \,\exists N\le \Phi(k) \,\,  (x_N\in AF_k).
\eeq
If $\Phi$ is an approximate $F$-point bound for $(x_n)$, then 
$$\Phi^M:\N\to\N, \quad \Phi^M(k)=\max\{\Phi(m)\mid m\leq k\}$$
is monotone nondecreasing and again an approximate $F$-point bound for $(x_n)$. 

{\em Thus, we shall assume w.l.o.g. that any approximate $F$-point bound for $(x_n)$ is monotone nondecreasing.} 

Then $\Phi:\N\to\N$ is an approximate $F$-point 
bound for $(x_n)$ if and only if $\Phi$ majorizes $\varphi_F$.\\

The next theorem is the main step towards a quantitative version of  Proposition  \ref{general-GH-fejer} (see also the discussion in
\cite[pp. 464-465]{Kohlenbach(book)} on the logical background behind
the elimination of sequential compactness in the original proof
in favor of a computational argument):

\begin{theorem}\label{main-step-general-GH-fejer}
Assume that
\be
\item $(x_n)$ is uniformly $(G,H)$-Fej\'er monotone w.r.t. $F$,  with modulus $\chi$; 
\item $(x_n)$ has approximate $F$-points, with $\Phi$ being an approximate 
$F$-point bound.
\ee
Then $(x_n)$ is Cauchy and, moreover, for all $k\in\N$ and all $g:\N\to\N$,
\[\exists N\le \Psi(k,g,\Phi,\chi,\alpha_G,\beta_H,\gamma) 
\,\forall i,j\in [N,N+g(N)]\ 
\left(d(x_i,x_j)\le \frac1{k+1}\right), \]
where $\Psi(k,g,\Phi,\chi,\alpha_G,\beta_H,\gamma):=\Psi_0(P,k,g,\Phi,\chi,\beta_H)$, with 
\[ \chi_g(n,k):=\chi(n,g(n),k),\quad  \chi^M_g(n,k):=\max\{ \chi_g(i,k) \mid i\le n\},\]
$P:=\gamma\left( \alpha_G\left(2\beta_H(2k+1)+1\right)\right)$ 
and
\[ \left\{ \ba{l} \Psi_0(0,k,g,\Phi,\chi,\beta_H):=0 \\ 
\Psi_0(n+1,k,g,\Phi,\chi,\beta_H):=\Phi\left(\chi^M_g\left(\Psi_0(n,k,g,\Phi,\chi,\beta_H),2\beta_H(2k+1)+1\right)\right). \ea \right. \]  
\end{theorem} 
\begin{proof} Let $k\in\N$ and $g:\N\to\N$. For simplicity, let us denote with $\varphi$ the mapping $\varphi_F$ defined by \eqref{def-phi-F}. 
Since both $\varphi$ and $\Phi$ are nondecreasing and $\Phi$ majorizes $\varphi$, an immediate induction gives us that
$\Psi_0(n,k,g,\varphi,\chi,\beta_H)\leq \Psi_0(n+1,k,g,\varphi,\chi,\beta_H)$,  $\Psi_0(n,k,g,\Phi,\chi,\beta_H)\leq \Psi_0(n+1,k,g,\Phi,\chi,\beta_H)$ and
$\Psi_0(n,k,g,\varphi,\chi,\beta_H) \leq \Psi_0(n,k,g,\Phi,\chi,\beta_H)$ for all $n\in\N$. 

Define  for every $i\in\NN$ 
\beq 
n_i:=\Psi_0(i,k,g,\varphi,\chi,\beta_H). 
\eeq

{\bf Claim 1:} For all $j\ge 1$ and all $0\le i<j$,  $x_{n_j}$ is a $\chi_g(n_i,2\beta_H(2k+1)+1) $-approximate $F$-point.\\

{\bf Proof of claim:}  As $j\ge 1$ and
\bua 
n_j&=& \Psi_0(j,k,g,\varphi,\chi,\beta_H)= \varphi\left(\chi^M_g\left(\Psi_0(j-1,k,g,\varphi,\chi,\beta_H),2\beta_H(2k+1)+1\right)\right)\\
 &=& \varphi\left(\chi^M_g\left(n_{j-1},2\beta_H(2k+1)+1\right)\right),
 \eua
 $x_{n_j}$ is a $\chi^M_g(n_{j-1}, 2\beta_H(2k+1)+1)$-approximate $F$-point. Since $0\le i\le j-1$, we have that $n_i\leq n_{j-1}$. Apply now the fact that $\chi_g^M$ is nondecreasing in the first argument to get that
 \bua
  \chi_g(n_i, 2\beta_H(2k+1)+1) 
&\le &  \chi^M_g(n_i, 2\beta_H(2k+1)+1)\\
&\le &  \chi^M_g(n_{j-1}, 2\beta_H(2k+1)+1). \,\,\, \qquad \hfill \blacksquare
\eua

{\bf Claim 2:} There exist  $\ds 0\leq I< J \le P$ satisfying 
\[\forall l\in [n_I,n_I+g(n_I)] \, \left( d(x_l,x_{n_J})\le \frac1{2k+2}\right). \]
{\bf Proof of claim:}  By the property 
of $\gamma$ being a II-modulus of total boundedness for $X$ we get that there exist  $\ds 0\leq I< J \le P$ such that 
\[d(x_{n_I},x_{n_J})\le \frac1{\alpha_G(2\beta_H(2k+1)+1)+1}\] 
and so, using that $\alpha_G$ is a $G$-modulus, 
\beq
G(d(x_{n_I},x_{n_J}))\le \frac1{2\beta_H(2k+1)+2}. \label{q-thm-1-useful-1}
\eeq
By the first claim, we have that $x_{n_J}$ is a  
$\ds \chi_g(n_I,2\beta_H(2k+1)+1)$-approximate $F$-point.  Applying now  the uniform $(G,H)$-F\'{e}jer monotonicity of $(x_n)$  w.r.t. $F$ with 
$\ds r:=2\beta_H(2k+1)+1, n:=n_I, m:=g(n_I)$ and  $p:=x_{n_J}$,  we get that for all $l\le g(n_I)$,
\bua
H(d(x_{n_I+l},x_{n_J})) & \le & G(d(x_{n_I},x_{n_J}))+\frac1{2\beta_H(2k+1)+2} \le \frac1{\beta_H(2k+1)+1}.
\eua
\details{We have that $x_{n_J}$ is a $\ds \chi_g(n_I,2\beta_H(2k+1)+1)$-approximate $F$-point and 
$\ds \chi_g(n_I,2\beta_H(2k+1)+1)=\chi(n_I,g(n_I), 2\beta_H(2k+1)+1)$
}
Since $\beta_H$ is an $H$-modulus, 
\[ \forall l\le g(n_I) \ \left( d(x_{n_I+l},x_{n_J})\le \frac1{2k+2} \right).\]
 and so the claim is proved. \hfill $\blacksquare$\\[1mm]
It follows that 
\[ \forall k,l\in [n_I,n_I+g(n_I)] \ \left( d(x_k,x_l)\le \frac1{k+1}\right). \]

Since $n_I=\Psi_0(I,k,g,\varphi,\chi,\beta_H)\leq \Psi_0(I,k,g,\Phi,\chi,\beta_H)$ and $I\leq P$, we  get that
\[n_I\le \Psi_0(P,k,g,\Phi,\chi,\beta_H)=\Psi(k,g,\Phi,\chi,\alpha_G,\beta_H,\gamma). \]
The theorem holds with $N:=n_I$. 
\end{proof} 
{\bf Corollary to the proof:} One of the numbers $n_0,\ldots,n_{P-1}$ is a point 
of metastability. \\[-1mm]
\details{
\[
\Psi:\N\times\N^\N\times\N^\N\times \N^{N\times\N\times\N}\times\N^\N\times\N^\N\times\N^\N\to \N\]
defined by 
\beq \Psi(m,h,\Phi,\theta,\alpha,\beta,\gamma):=\Psi_0(P,m,h,\Phi,\theta,\beta),  \label{definition-Psi-general} 
\eeq
with $P=\gamma\left( \alpha\left(2\beta(2m+1)+1\right)\right), \theta_h(n,m):=\theta(n,h(n),m), \theta^M_h(n,m):=\max\{ \theta_h(i,m) \mid i\le n\}$ and
\[ \left\{ \ba{l} \Psi_0(0,m,h,\Phi,\theta,\beta):=0 \\ 
\Psi_0(n+1,m,h,,\Phi,\theta,\beta):=\Phi\left(\theta^M_h\left(\Psi_0(n,m,h,\Phi,\theta,\beta),2\beta(2m+1)+1\right)\right). \ea \right. \]  
The bound from theorem is $\Psi(k,g,\Phi,\chi,\alpha_G, \beta_H,\gamma)$.
}

Theorem \ref{main-step-general-GH-fejer} remarkably implies the 
Cauchy property of $(x_n)$ in the absence of $X$ being complete (and hence 
compact) and of $F$ being explicitly closed which, as we remarked after 
Proposition \ref{general-GH-fejer}, both were necessary if $(x_n)$ only 
was assumed to be $(G,H)$-Fej\'er monotone rather than being 
uniformly $(G,H)$-Fej\'er monotone. This is  a {\bf qualitative} 
improvement of Proposition \ref{general-GH-fejer} whose proof is  
based on our quantitative analysis of metastability although the result as 
such does not involve metastability at all:

\begin{corollary} \label{qualitative-corollary}
Let $X$ be totally bounded and $(x_n)$ be uniformly $(G,H)$-Fej\'er monotone 
having approximate $F$-points. Then $(x_n)$ is Cauchy.
\end{corollary}

The next theorem is a direct quantitative `finitization' of Proposition \ref{general-GH-fejer}  in the sense of Tao:

\begin{theorem}\label{finitization-general-GH-Fejer} 
In addition to the assumptions of Theorem \ref{main-step-general-GH-fejer} we 
suppose that $F$ is uniformly closed with 
moduli $\delta_F,\omega_F.$ Then for all $k\in\N$ and all $g:\N\to\N$,
\[  \exists N\le 
\tilde{\Psi} 
\,\forall i,j\in [N,N+g(N)]\
 \left( d(x_i,x_j)\le \frac1{k+1} \text{~and~} x_i\in AF_k\right), \]
where 
\[ \tilde{\Psi}\left( k,g,\Phi,\chi,\alpha_G,\beta_H,\gamma, \delta_F, \omega_F\right):=
\Psi(k_0,g,\Phi,\chi_{k,\delta_F},\alpha_G, \beta_H,\gamma), \] with $\Psi$ defined as in Theorem \ref{main-step-general-GH-fejer}    
\[ k_0=\max\left\{ k,\left\lceil\frac{\omega_F(k)-1}{2}\right\rceil \right\} \text{~and~} \chi_{k,\delta_F}(n,m,r):=
\max\{\delta_F(k),\chi(n,m,r)\}.\]
\end{theorem}
 \begin{proof}
  With $\chi$ also $\chi_{k,\delta_F}$ is a modulus of $(x_n)$ being uniformly $(G,H)$-Fej\'er 
monotone w.r.t. $F$.  Applying Theorem \ref{main-step-general-GH-fejer} to $(k_0,\chi_{k,\delta_F})$ we get that 
\[ \exists N\le 
\tilde{\Psi} 
\,\forall i,j\in [N,N+g(N)]\
 \left(d(x_i,x_j)\le \frac1{k_0+1}\le \frac1{k+1}\right). \]
 
 From the proof of Theorem \ref{main-step-general-GH-fejer}, it follows that there exists $0\le I<J\le P$ such that  $N=n_I$ and $x_{n_J}$ is a  
$(\chi_{k,\delta_F})_g(N, 2\beta_H(2k_0+1)+1)$-approximate $F$-point and 
 \[ \forall i\in [N,N+g(N)]\ \left( d(x_i,x_{n_J})\le  \frac1{2k_0+2} \le \frac1{\omega_F(k)+1}\right).\]
 Since $(\chi_{k,\delta_F})_g(N,2\beta_H(2k_0+1)+1) =\chi_{k,\delta_F}(N,g(N),2\beta_H(2k_0+1)+1) \geq \delta_F(k)$,  it follows that $x_{n_I}$ is a 
$\delta_F(k)$-approximate $F$-point. Hence by the definition of $\omega_F$, we get that $x_i\in AF_k$ for all  $i\in [N,N+g(N)]$.
  \end{proof}
  \details{The bound is $\Psi(k_0,g,\Phi,\chi_{k,\delta_F},\alpha_G, \beta_H,\gamma))$ where $\Psi$ is defined by \eqref{definition-Psi-general} .
  }
{\bf Notation:} 
In our applications $\delta_F$ will be mostly $\delta_F(k)=2k+1$.  In this 
case we simply write $\chi_k$ instead of $\chi_{k,\delta_F}$ when applying 
Theorem \ref{finitization-general-GH-Fejer}.  
  
\begin{remark}\label{finitization-boundedly-compact}
Theorems \ref{main-step-general-GH-fejer}  and \ref{finitization-general-GH-Fejer}  hold for $X$ boundedly compact and $(x_n)$ bounded. 
In this case, the bounds  will depend on a II-modulus of total boundedness for the 
closed ball $\ol{B}(a,b)$, where $a\in X$ and $b\geq d(x_n,a)$ for all $n$.
\end{remark}

\begin{remark} 
Theorem \ref{finitization-general-GH-Fejer}  is a finitization of Proposition \ref{general-GH-fejer} in the sense of Tao since it only talks about a finite 
initial segment of $(x_n)$ but trivially implies back the infinitary Proposition \ref{general-GH-fejer} for uniformly closed $F$ and 
uniformly $(G,H)$-Fej\'er monotone sequences.
\end{remark}
\begin{proof}  Noneffectively
\[
\forall k\in\NN\,\forall g:\NN\!\to\!\NN\,\exists N\in\NN 
\,\forall i,j\in [N,N+g(N)]\,
 \left( d(x_i,x_j)\le \frac1{k+1} \text{~and~} x_i\in AF_k \right) 
\]
implies the Cauchy property of $(x_n)$. Since $X$ is complete, $(x_n)$ converges   to a point $\widehat{x}\in X$. 
It remains to prove that $\widehat{x}\in F$. One can easily see, by taking 
$g$ to be a constant function, that $(x_n)$ has the liminf property w.r.t. $F$. 
\details{Let $k,n\in\N$. By taking $g_n:\N\to\N, g_n(m)=n$ for all $m\in\N$, we get $N\in\N$ such that $x_i\in AF_k$ for all $i\in[N,N+n]$. 
In particular, $x_{N+n}\in AF_k$ and $N+n\geq n$.}
Apply now Lemma \ref{F-unifcont-liminf-F} to conclude that  $\widehat{x}\in F$.
  \end{proof}

Assume that $(x_n)$ is asymptotically regular w.r.t. $F$.  A mapping $\Phi^+:\N\times \N^\N\to\N$ satisfying
\[\forall k\in\NN\,\forall g:\NN\to\NN\,\exists N\le\Phi^+(k,g) 
\,\forall m\in [N,N+g(N)]  \,\, (x_m\in AF_k)\]
is said to be a {\em rate of metastability} for the asymptotic regularity of $(x_n)$ w.r.t. $F$. A {\em rate} of asymptotic regularity of $(x_n)$ w.r.t. $F$  is a function $\Phi^{++}:\NN\to\NN$ with
\[
\forall k\in\NN \,\exists N\le \Phi^{++}(k) \,  \forall m\ge N \,\,   (x_m\in AF_k)
\]
Obviously, this is equivalent with the fact that $\Phi^{++}$ satisfies
\[
\forall k\in\NN \,\forall n \ge \Phi^{++}(k) \, \, (x_n\in AF_k). 
\]

If instead of an approximate $F$-point bound $\Phi$ for $(x_n)$ one has a rate of metastability $\Phi^+$ for the asymptotic regularity of $(x_n)$ w.r.t. $F$, then one can directly combine $\Psi$ from Theorem \ref{main-step-general-GH-fejer} and 
such a $\Phi^+$ into a bound $\Psi':=\Omega(\Psi,\Phi^+)$ satisfying the claim of Theorem \ref{finitization-general-GH-Fejer}  
without any uniform closedness assumption on $F$ and no need to use $\omega_F$.  The transformation $\Omega$ gets particularly simple if
instead of $\Phi^+$ we even have a rate $\Phi^{++}$ of asymptotic regularity w.r.t. $F$.\\

We first have to define one more case of the general majorization relation $\gtrsim$:

\begin{definition} A functional 
$\Phi:\N\times\N^\N\to \N$ is majorized by $\Phi^*:\N\times\N^\N\to\N$ 
(short $\Phi^* \gtrsim \Phi$)  if 
\[ \forall k\in\N\, \forall g:\N\to\N \,\, \left(k'\ge k \text{~and~}  g'\gtrsim_{\NN\to\NN} g\to \Phi^*(k',g')\ge \Phi^*(k,g),\Phi(k,g)\right).\] 
$\Phi$ is called selfmajorizing if $\Phi \gtrsim \Phi.$
\end{definition}

For any function $f:\N\to\N$ we shall denote 
$$f^M:\N\to\N, \quad f^M(n):=\max\{f(i)\mid i\leq n\}.$$
Then, as remarked in Section \ref{logical-meta-tb-metric}, $f^M\gtrsim_{\N\to\N} f$.\\

In the sequel, $k\in\N$ and $g:\N\to\N$. We define the following functionals:

\be
\item $g^*:\N\to\N, \quad g^*(n)=n+g^M(n)$;
\item for every $l\in\N$,\,\, $\tilde{g}_l:\N\to\N, \quad \tilde{g}_l(m):=g^*(\max\{l,m\})$;
\item for every $\delta:\N\times\N^\N\to \N$, 
\beq
h_{k,g,\delta}:\N\to \N, \quad h_{k,g,\delta}(n):= g^*(\max\{n,\delta(k, \tilde{g}_n)\}).
\eeq
\item $\Omega_{k,g}:(\N\times\N^\N\to \N)\times (\N\times\N^\N\to \N)\to \N$, defined for every $\delta,\theta:\N\times\N^\N\to \N$, by 
\beq
\Omega_{k,g}(\delta,\theta):=\max\{\delta(k,h_{k,g,\theta}), \theta(k,\tilde{g}_{\delta(k,h_{k,g,\theta})})\};
\eeq
\item for every $l\in\N$,\,\, $g_l(n):\N\to\N, \quad g_l(n):=g^M(n+l)+l$;
\item $\tilde{\Omega}_{k,g}:(\N\times\N^\N\to \N)\times \N^\N\to \N$, defined for every $\delta:\N\times\N^\N\to \N, \, f:\N\to\N$, by 
\beq
\tilde{\Omega}_{k,g}(\delta,f):= \delta(k,g_{f(k)})+f(k).
\eeq
\ee

One can easily verify the following 

\begin{lemma}\label{majorant-prop-NNN}
\be
\item For all $l,l^*\in\N$, $l^*\geq l$ implies $\tilde{g}_{l^*}\gtrsim \tilde{g}_{l}$ and  $g_{l^*}\gtrsim g_l$.
\item For all $\delta,\delta^*:\N\times\N^\N\to \N$, $\delta^*\gtrsim \delta$ implies $h_{k,g,\delta^*}\gtrsim h_{k,g,\delta}$.
\item\label{majorant-prop-NNN-omega} For all $\delta,\delta^*,\theta,\theta^*:\N\times\N^\N\to \N$,  if $\delta^*\gtrsim \delta$ and $\theta^*\gtrsim \theta$, then $\Omega_{k,g}(\delta^*,\theta^*)\geq \Omega_{k,g}(\delta,\theta)$.
\item\label{majorant-prop-NNN-omega-tilde} For all $\delta,\delta^*:\N\times\N^\N\to \N$, $f,f^*:\N\to\N$ if $\delta^*\gtrsim \delta$ and $f^*\gtrsim f$, 
then $\tilde{\Omega}_{k,g}(\delta^*,f^*)\geq \tilde{\Omega}_{k,g}(\delta,f)$.
\ee
\end{lemma}
\details{
\be
\item Let $l^*\geq l$. We have that $\tilde{g}_{l^*}\gtrsim_{\N\to\N} \tilde{g}_{l}$

\begin{tabular}{l}
  iff for every $n^*\geq n$, $\tilde{g}_{l^*}(n^*)\geq \tilde{g}_{l^*}(n), \tilde{g}_{l}(n)$\\
 iff for every $n^*\geq n$, $g^*(\max\{l^*,n^*\})\geq g^*(\max\{l^*,n\}), g^*(\max\{l,n\})$,
\end{tabular}

We have that $g_{l^*}\gtrsim_{\N\to\N} g_{l}$

\begin{tabular}{l}
  iff for every $n^*\geq n$, $g_{l^*}(n^*)\geq g_{l^*}(n), g_{l}(n)$\\
 iff for every $n^*\geq n$, $g^M(n^*+l^*)+l^*\geq g^M(n+l^*)+l^*, g^M(n+l)+l$,
\end{tabular}
which is true.
\item We have that $h_{k,g,\delta^*}\gtrsim_{\N\to\N} h_{k,g,\delta}$

\begin{tabular}{l}
  iff for every $n^*\geq n$,  $h_{k,g,\delta^*}(n^*)\geq h_{k,g,\delta^*}(n), h_{k,g,\delta}(n)$\\
 iff for every $n^*\geq n$, 
$g^*(\max\{n^*,\delta^*(k, \tilde{g}_{n^*})\} \geq g^*(\max\{n,\delta^*(k, \tilde{g}_{n})\}, g^*(\max\{n,\delta(k, \tilde{g}_{n})\}$  
\end{tabular}
Since $\delta^*\gtrsim \delta$, $k\ge k$ and $\tilde{g}_{n^*}\gtrsim_{\N\to\N}\tilde{g}_{n}$, it follows that 
\[\delta^*(k, \tilde{g}_{n^*})\geq \delta^*(k, \tilde{g}_{n}), \delta(k, \tilde{g}_{n}).\]
Use again that  $g^*$ is nondecreasing. \
\item Assume that $\delta^*\gtrsim \delta$ and $\theta^*\gtrsim \theta$. It follows that 
 $h_{k,g,\theta^*}\gtrsim_{\N\to\N} h_{k,g,\theta}$, hence $\delta^*(k,h_{k,g,\theta^*})\geq \delta(k,h_{k,g,\theta})$. Furthermore,
$\tilde{g}_{\delta^*(k,h_{k,g,\theta^*})} \gtrsim_{\N\to\N} \tilde{g}_{\delta(k,h_{k,g,\theta})}$, hence 
$\theta^*(k,\tilde{g}_{\delta^*(k,h_{k,g,\theta^*})})\geq \theta(k,\tilde{g}_{\delta(k,h_{k,g,\theta})})$. 
\item   Assume that $\delta^*\gtrsim \delta$ and $f^*\gtrsim f$. Then $f^*(k)\geq f(k)$. Furthermore, $g_{f^*(k)}\gtrsim_{\N\to\N}$. Hence
$\delta^*(k,g_{f^*(k)})\geq \delta(k,g_{f(k)})$. 
\ee
}

\begin{theorem}\label{Cauchy+as-reg-metastability}
Let $(x_n)$ be a Cauchy sequence with a selfmajorizing rate of metastability $\Psi$. 
\be
\item\label{Carm-meta} Assume that $(x_n)$ is asymptotically regular w.r.t. $F$, with $\Phi^+$ being a selfmajorizing rate of metastability for the asymptotic regularity. 
Then for all $k\in\N$ and all $g:\N\to\N$,
\[\exists N\le \Omega
\,\forall i,j\in [N,N+g(N)]\
 \left( d(x_i,x_j)\le \frac1{k+1} \text{~and~} x_i\in AF_k\right), \]
where $\Omega(k,g,\Psi,\Phi^+):=\Omega_{k,g}(\Psi,\Phi^+)$.
\item\label{Carm-asreg} Assume that $(x_n)$ is asymptotically regular w.r.t. $F$, with $\Phi^{++}$ being a rate of asymptotic regularity.
Then for all $k\in\N$ and all $g:\N\to\N$,
\[ \exists N\le \tilde{\Omega}
\,\forall i,j\in [N,N+g(N)]\,\forall m\ge N\
 \left( d(x_i,x_j)\le \frac1{k+1}  \text{~and~}  x_m\in AF_k\right),\]  
where $\tilde{\Omega}(k,g,\Psi,\Phi^{++}):=\tilde{\Omega}_{k,g}(\Psi,(\Phi^{++})^M)$.
\ee
\end{theorem}
\begin{proof}   
\be
\item Let $\psi, \varphi^+:\N\times \N^\N\to\N$ 
be the functionals which, given $l\in\N$ and $\delta:\N\to\N$,  search for the least 
actual point $n$ of metastability upper bounded by $\Psi(l,\delta),\Phi^+(l,\delta)$ 
respectively. Thus, we have that $\psi(l,\delta)\leq \Psi(l,\delta)$, $\varphi^+(l,\delta)\leq\Phi^+(l,\delta)$,
\beq
\forall i,j\in [\psi(l,\delta),\psi(l,\delta)+\delta(\psi(l,\delta))]\,\,
     \left( d(x_i,x_j)\le \frac1{l+1} \right) \label{thm-self-1}
\eeq
and
\beq
\forall i\in [\varphi^+(l,\delta),\varphi^+(l,\delta)+\delta(\varphi^+(l,\delta))]\,\,
     \left( x_i\in AF_l\right). \label{thm-self-2}
\eeq
Let us take 
\[N:=\Omega_{k,g}(\psi,\varphi^+).\]
{\bf Claim:} \[ \forall i,j\in [N,N+g(N)]\, 
 \left( d(x_i,x_j)\le \frac1{k+1} \text{~and~} x_i\in AF_k\right). \]
{\bf Proof of claim:} Let $N_0:=\psi(k,h_{k,g,\varphi^+})$. Then 
$N=\max\{N_0,\varphi^+(k,\tilde{g}_{N_0})\}.$
Apply \eqref{thm-self-1} with $l:=k$ and $\delta:=h_{k,g,\varphi^+}$ and \eqref{thm-self-2} to $l:=k$ and $\delta:=\tilde{g}_{N_0}$ to get that
\beq
\forall i,j\in [N_0,N_0+h_{k,g,\varphi^+}(N_0)]\,\,
     \left( d(x_i,x_j)\le \frac1{k+1} \right) \label{thm-self-1star}
\eeq
and
\beq
\forall i\in [\varphi^+(k,\tilde{g}_{N_0}),\varphi^+(k,\tilde{g}_{N_0})+\tilde{g}_{N_0}(\varphi^+(k,\tilde{g}_{N_0}))]\,\,
     \left( x_i\in AF_k\right). \label{thm-self-2star}
\eeq
Remark now that  $N$ is an upper bound for both $N_0$ and $\varphi^+(k,\tilde{g}_{N_0})$ and, furthermore, that
\bua
h_{k,g,\varphi^+}(N_0) &=& \tilde{g}_{N_0}(\varphi^+(k,\tilde{g}_{N_0}))=  g^*(\Omega_{k,g}(\psi,\varphi^+))\\
&\geq & \Omega_{k,g}(\psi,\varphi^+)+g(\Omega_{k,g}(\psi,\varphi^+))=N+g(N).
\eua
The claim follows. $\qquad \hfill \blacksquare$

Since $\Psi,\Phi^+$ are selfmajorizing and bounds for $\psi,\phi^+$ they 
are majorants for $\psi,\phi^+$. Apply now Lemma \ref{majorant-prop-NNN}.\eqref{majorant-prop-NNN-omega} to conclude that 
$N=\Omega_{k,g}(\psi,\varphi^+)\leq \Omega_{k,g}(\tilde{\Psi},\Phi^+)(g)$.
\item Let $\psi:\N\times \N^\N\to\N$ be as in (i) and take 
\[N:=\tilde{\Omega}_{k,g}(\psi,\Phi^{++}).\]

Apply \eqref{thm-self-1} with 
$l:=k$ and $\delta:=g_{\Phi^{++}(k)}$  and denote $N_0:=\psi(k,g_{\Phi^{++}(k)})$ to get that
\[
\forall i,j\in [N_0,N_0+g_{\Phi^{++}(k)}(N_0)]\,\,
     \left( d(x_i,x_j)\le \frac1{k+1} \right)
\]
Remark that $N_0\leq N$ and that
\bua
N_0+g_{\Phi^{++}(k)}(N_0) &=& \tilde{\Omega}(k,g,\psi,\Phi^{++})+g^M(\tilde{\Omega}(k,g,\psi,\Phi^{++}))\\
& = & N+g^M(N) \geq  N+g(N).
\eua
Furthermore, since $N\geq \Phi^{++}(k)$ and  $\Phi^{++}$ is a rate of asymptotic regularity w.r.t. $F$, we get that $x_m\in AF_k$ for all $m\geq N$.
The fact that $N\leq \tilde{\Omega}_{k,g}(\Psi,(\Phi^{++})^M)$ follows immediately, using Lemma \ref{majorant-prop-NNN}.\eqref{majorant-prop-NNN-omega-tilde}. 
\ee
\end{proof}

In fact, one may also swap the roles of $\Psi$ and $\Phi^+$ in the definition of $\Omega$ and, in practice, one  has to check which one results 
in a better bound. Furthermore, the assumption that $\Psi,\Phi^+$ are selfmajorizing can always been achieved for bounds 
$\Psi,\Phi^+$ extracted via the proof-theoretic methods presented in Section \ref{logic-section}.

\begin{corollary}\label{cor-self-maj}
Let $(x_n)$ be a uniformly $(G,H)$-Fej\'er monotone sequence with 
modulus $\chi$ which is asymptotically regular with a selfmajorizing 
rate of metastability 
for the asymptotic regularity $\Phi^+.$ 
For each, $\alpha_G,\beta_H,\gamma :\N\to\N$ and $\chi:\N^3\to\N$ 
define the functional 
\[ \Psi^+(k,g):=\Psi(k,g,\Phi,\chi^M,\alpha^M_G,\beta^M_H,\gamma^M),\] where 
$\Psi$ is the bound from 
Theorem \ref{main-step-general-GH-fejer}, $\Phi(k):=\Phi^+(k,0)$ and
\[\chi^M(n,m,k):=\max\{ \chi(\tilde{n},\tilde{m},\tilde{k}) \mid
\tilde{n}\le n, \tilde{m}\le m, \tilde{k}\le k\}).\]
Then for all $k\in\N, g:\N\to\N$ 
\[ \exists N\le\Omega(k,g,\Psi^+,\Phi^+)\,\forall i,j\in [N,N+g(N)] \ 
\left(d(x_i,x_j)\le \frac1{k+1}\ \mbox{and} \ x_i \in AF_k\right). \]
Similarly for $\tilde{\Omega}$ instead of $\Omega$ (with $\Phi:=(\Phi^{++})^M$) 
where we then even have $\forall m\ge N (x_m\in AF_k).$
\end{corollary}
\begin{proof} With $\chi,\alpha_G,\beta_H,\gamma$ also $\chi^M,\alpha^M_G,\beta^M_H,\gamma^M$ 
are a modulus of uniform $(G,H)$-Fej\'er monotonicity, $G,H$-moduli and a 
II-modulus of total boundedness, respectively. Moreover, $\Phi$ is an approximate $F$-point 
bound for $(x_n).$ 
Hence, by Theorem \ref{main-step-general-GH-fejer}, $\Psi^+$ is a rate of 
metastability for $(x_n)$ which, moreover, is selfmajorizing. The claim now 
follows from Theorem 
\ref{Cauchy+as-reg-metastability}.
\end{proof}
We conclude this section with a trivial but instructive example for 
Theorem \ref{main-step-general-GH-fejer}, 
namely that the well-known rate of metastability for 
the Cauchy property of monotone bounded sequences from \cite[Proposition 2.27]{Kohlenbach(book)} 
can be recovered (modulo a constant) from this theorem:
let $X=[0,1]$ and $(x_n)$ be a nondecreasing sequence in $X$. Let us take
\[
F= \bigcap_{k\in\N}\tilde{F}_k, \quad \text{where } \tilde{F}_k:=\{p\in X \mid x_k\le p\}.
\]
Then, clearly, (i) $AF_k=\tilde{F}_k,$ (ii) $\Phi^{++}:=id$ 
is a rate of asymptotic regularity and $\chi(n,m,r):=n+m$ is a modulus 
of the uniform Fej\'er monotonicity of $(x_n)$ and we may take 
$\gamma(k)= k+1$ (see Example \ref{ex-tb-01}). 
For monotone $g,$ Theorem \ref{main-step-general-GH-fejer} now gives 
$\Psi(k,g):=\tilde{g}^{4(k+1)}(0)$ with $\tilde{g}(n):=n+g(n),$ while the 
direct proof in this case yields $\Psi(k,g):=\tilde{g}^{k}(0).$ 

\section{Quasi-Fej\'er monotone sequences} \label{section-quasi}

As a common consequence of arriving at a finitary quantitative version of 
an originally non-quan-titative theorem, one can easily incorporate error 
terms as has been considered under the name of quasi-Fej\'er monotonicity 
(due to \cite{Ermolev}). As pointed out in \cite{Com01}, quasi-Fej\'er monotone sequences  provide a framework  for the analysis of numerous optimization algorithms in Hilbert spaces.

\begin{definition} A sequence $(x_n)$ in a metric space $(X,d)$ is called 
quasi-Fej\'er monotone (of order $0<P<\infty$) 
w.r.t. some set $\emptyset\not= F\subseteq X$ if 
\[ \forall n\in\N\,\forall p\in F\ \big( d(x_{n+1},p)^P \le d(x_n,p)^P +
\varepsilon_n\big), \] 
where $(\varepsilon_n)$ is some summable sequence in $\R_+$.
\end{definition} 

The appropriate generalization to general functions $(G,H)$ then is:
\begin{definition} For $G,H$ as in the definition of $(G,H)$-Fej\'er 
monotonicity we say that $(x_n)$ is quasi-$(G,H)$-Fej\'er monotone w.r.t. $F$ if 
\[ \forall n,m\in\N\,\forall p\in F\ \big( H(d(x_{n+m},p))\le G(d(x_n,p))+
\sum_{i=n}^{n+m-1} \varepsilon_i\big).\] 
\end{definition} 
Note that for $G(x):=H(x):=x^P$ this covers the notion of quasi-Fej\'er monotonicity.
\\ The uniform version of this notion then is:
\begin{definition} 
$(x_n)$ is uniformly quasi-$(G,H)$-Fej\'er monotone w.r.t. $F$ and a given 
representation of $F$ via $AF_k$ as before if 
\[ \ba{l} \forall r,n,m\in\N\,\exists k\in\N\,\forall p\in X \,\big( 
p\in AF_k\to \\ \hspace*{1cm} 
\forall l\le m (H(d(x_{n+l},p))<G(d(x_n,p))+\sum_{i=n}^{n+m-1} \varepsilon_i +\frac{1}{r+1})\big).\ea \]
Any function $\chi:\N^3\to \N$ such that $\chi(r,n,m)$ provides such a $k$ is 
called a modulus of $(x_n)$ being uniformly quasi-$(G,H)$-Fej\'er monotone w.r.t. $F.$
\end{definition}
Let $\xi:\N\to\N$ be a Cauchy modulus of $\sum\varepsilon_i,$ 
i.e. $\sum\limits_{i=\xi(n)}^{\infty} \varepsilon_i <\frac{1}{n+1}$ for all $n\in\N.$ 
\\ If $(x_n)$ has the $\liminf$-property w.r.t. $F$ we can define 
\[ \widehat{\varphi}_F(k,n):=\min\{ m\in\N \mid m\ge n\wedge x_m\in AF_k\}. \]
Any monotone (in $k,n$) upper bound $\widehat{\Phi}$ of 
$\widehat{\varphi}_F$ is 
called a $\liminf$-bound w.r.t. $F.$
\begin{theorem}\label{main-step-general-quasi-GH-fejer}
Assume that
\be
\item $(x_n)$ is uniformly quasi-$(G,H)$-Fej\'er monotone w.r.t. $F$,  with modulus $\chi$, and $(\varepsilon_n)$ with Cauchy rate $\xi$ for 
$\sum \varepsilon_i$; 
\item $(x_n)$ has the $\liminf$-property w.r.t. $F$, with $\widehat{\Phi}$ 
being a $\liminf$-bound w.r.t. $F$.
\ee
Then $(x_n)$ is Cauchy and, moreover, for all $k\in\N$ and all $g:\N\to\N$,
\[\exists N\le \widehat{\Psi}(k,g,\Phi,\chi,\alpha_G,\beta_H,\gamma,\xi) 
\,\forall i,j\in [N,N+g(N)]\ 
\left(d(x_i,x_j)\le \frac1{k+1}\right), \]
where $\widehat{\Psi}(k,g,\Phi,\chi,\alpha_G,\beta_H,\gamma,\xi):=
\widehat{\Psi}_0(P,k,g,\Phi,\chi,\beta_H,\xi)$, with 
\[ \chi_g(n,k):=\chi(n,g(n),k),\quad  \chi^M_g(n,k):=\max\{ \chi_g(i,k) \mid i\le n\},\]
$P:=\gamma\left( \alpha_G\left(4\beta_H(2k+1)+3\right)\right)+1$ 
and
\[ \left\{ \ba{l} \widehat{\Psi}_0(0,k,g,\Phi,\chi,\beta_H,\xi):=0 \\ 
\widehat{\Psi}_0(n+1,k,g,\Phi,\chi,\beta_H,\xi):= \\ 
\hspace*{1cm} \widehat{\Phi}\left(\chi^M_g\left(\Psi_0(n,k,g,\Phi,\chi,\beta_H,\xi),4\beta_H(2k+1)+3\right),\xi(4\beta_H(2k+1)+3)\right). \ea \right. \]  
\end{theorem} 
\begin{proof} The proof is the same as the one of Theorem 
\ref{main-step-general-GH-fejer} up to 
(\ref{q-thm-1-useful-1}) which now holds with $1/(4\beta_H(2k+1)+4)$ 
instead of $1/(2\beta_H(2k+1)+2).$ We then use uniform quasi-$(G,H)$-Fej\'er 
monotonicity as we did before without `quasi' to now get that for all 
$l\leq g(n_I)$ 
\[ H(d(x_{n_I +l},x_{n_J}))\le G(d(x_{n_I},x_{n_J}))+\sum^{n_I+l-1}_{i=n_I} 
\varepsilon_i +\frac{1}{4\beta_H(2k+1)+4}. \] 
By construction of $n_I$ we know that $n_I\ge \xi(4\beta_H(2k+1)+3)$ (not that 
by the addition of `$+1$' to the original definition of $P$ that was used 
in the proof of Theorem \ref{main-step-general-GH-fejer} 
$I,J$ can now be choosen so that $0<I<J\le P$ rather than only 
$0\le I<J\le P$) and so 
\[ \sum^{n_I+l-1}_{i=n_I} 
\varepsilon_i \le \frac{1}{4\beta_H(2k+1)+4}\] and so we get in 
total 
\[ H(d(x_{n_I +l},x_{n_J}))\le G(d(x_{n_I},x_{n_J})) +\frac{1}{2\beta_H(2k+1)+2} \]  from where we can finish the proof as before.
\end{proof}   
As it is clear from the proof above, one actually does not need a Cauchy modulus $\xi$ of the error-sum but only a rate of metastability.\\[1mm] 
With the new bound $\widehat{\Psi}$ from 
Theorem \ref{main-step-general-quasi-GH-fejer} instead of $\Psi$ all the other 
results of the previous section extend in the obvious way to the `quasi'-case.
As a consequence of this, we could incorporate also in the iterations 
considered in the rest of this paper error terms which we, however, will 
not carry out.

\section{Application - $F$ is $Fix(T)$}\label{section-F-FT}

Let $X$ be a metric space, $C\se X$ a nonempty subset and $T:C\to C$ be a mapping. We assume that $T$ has  fixed points and define $F$ as  
the nonempty fixed point set $Fix(T)$ of  $T$. 

One has
\[
F=  \bigcap_{k\in\N}\tilde{F}_k, \quad \text{where } \tilde{F}_k=\left\{ x\in C \mid d(x,Tx)\le \frac1{k+1}\right\}.
\]
In this case, for all $k\in\NN$  we have that that $AF_k=\tilde{F}_k$ and the $k$-approximate $F$-points are  precisely the 
$1/(k+1)$-approximate fixed points of $T$.

Let us recall that the mapping $T$ is uniformly continuous with modulus $\omega_T:\N\to\N$ if for all $k\in\N$ and all $p,q\in C$,
\[
d(p,q)\leq \frac1{\omega_T(k)+1}  \,\, \rightarrow  \,\, d(Tp,Tq)\leq  \frac1{k+1}.  
\]

One can see easily that the following properties hold.

\begin{lemma}\label{FixT-xn-basic}
Let $(x_n)$ be a sequence in $C$.
\be
\item $(x_n)$ has approximate $F$-points if and only if for all $k\in\NN$ there exists  $N\in \N$ such that  $\ds d(x_N,Tx_N)\le  \frac1{k+1}$. 
If this is the case, we say also that $(x_n)$ has 
approximate fixed points.
\item $(x_n)$ has the liminf property w.r.t. $F$ if and only if $\linfn d(x_n,Tx_n)=0$. 
\item $(x_n)$ is asymptotically regular w.r.t. $F$ if and only if $\limn d(x_n,Tx_n)=0$. 
\item If $T$ is continuous, then $F$ is explicitly closed.
\item\label{Tuc-Fuc} If $T$ is uniformly continuous with modulus $\omega_T$, then $F$ is uniformly closed with moduli $\ds \omega_F(k)=\max\{4k+3,\omega_T(4k+3)\}$ and 
$\delta_F(k)=2k+1.$ 
\ee
\end{lemma}
\details{\begin{proof}
\be
\item is obvious.
\item Since $(d(x_n,Tx_n))_{n\in\N}$ is a sequence of nonnegative reals, we have that $\lim \inf d(x_n,Tx_n)=0$ iff 
$(x_n)$ has a subsequence which converges to $0$ iff $(x_n)$ has the liminf property w.r.t. $F$.
\item $(x_n)$ is asymptotically regular w.r.t. $F$ iff $\forall k\in \NN\,\exists N\in\NN\,\forall m\ge N \,\,\, \left(d(x_m,Tx_m)\leq  \frac1{k+1}\right)$ iff  $\limn d(x_n,Tx_n)=0$.
\item Assume that $T$ is continuous. Then the mapping $d_T:C\to \RR_+, \,\, d_T(x)=d(x,Tx)$ is continuous too. Let $k\ge 1$ and $p\in C$. 
Since $d_T$ is continuous at $p$, there exists $N\in\NN$ such that for all $q\in C$ such that  $d(p,q)\le  \frac1{N+1}$ we have that 
$|d_T(q)-d_T(p)|\leq  \frac1{2(k+1)}$. Since  
\[d(p,Tp)=d_T(p)\leq d_T(q)+\frac1{2(k+1)} =d(q,Tq)+\frac1{2(k+1)} \]
it follows that 
\[ d(q,Tq)\le \frac1{2(k+1)}\to d(p,Tp)\le \frac1{k+1} \]
Take $M:=\frac1{2(k+1)}$. 
\item Assume that $T$ is uniformly continuous with modulus $\omega_T$. Let $k\ge 1$ and $p,q\in C$ be such that
$d(q,Tq)\le   \frac1{\delta_F(k)+1}=\frac1{2(k+1)}$  and $d(p,q)\le \frac1{\omega_F(k)+1}\le  \frac1{\omega_T(4k+3)+1},  \frac1{4k+4}$, 
by the definition of $\omega_F$.  We get that  $d(Tp,Tq)\leq \frac1{4k+4}$. Hence,
\bua
d(p,Tp) & \leq & d(p,q)+d(q,Tq)+d(Tq,Tp)\leq \frac1{4k+4}+\frac1{2k+2}+\frac1{4k+4}\\
&=& \frac1{k+1}.
\eua
\ee
\end{proof}
}

\begin{remark} Uniform closedness can be viewed as a quantitative version of the special extensionality statement $(*)\ q\in F\wedge p=_X q\to p\in F.$
Extensionality w.r.t. $p=_X q :=\| p-q\|_X =0$ is not included as an
axiom in our formal framework (for reasons explained in \cite{Koh05a})
and has to be derived (if needed) from appropriate uniform continuity
assumptions (see item $(v)$ in the lemma above). In the case of $(*)$,
however, it suffices to have the moduli $\omega_F,\delta_T$ which (as we will see in Section \ref{section-E}) are also available for interesting classes of in general discontinuous mappings $T$ (where, in particular, 
the model theoretic approach 
to metastability from \cite{AvigadIovino} is not applicable as is stands).
\end{remark}

As a consequence of Proposition \ref{general-GH-fejer} and Remark \ref{general-fejer-boundedly-compact}, we get 

\begin{proposition}\label{general-fejer-FT} 
Let $C$ be a boundedly compact subset of a metric space $X$ and 
$T:C\to C$ be continuous with $F=Fix(T)\ne\emptyset$. Assume that   $(x_n)$  is bounded and $(G,H)$-Fej\'er monotone with respect to $F$ and that  $(x_n)$ has approximate fixed points. Then $(x_n)$  converges to a fixed point of $T$. \\ 
The continuity of $T$ can be replaced by the weaker assumption that 
$F$ is explicitly closed (see Section \ref{section-E} for a class of in general 
discontinuous functions for which $Fix(T)$ is uniformly closed).
\end{proposition}

If we weaken `boundedly compact' to `totally bounded' or drop the assumption that $T$ is continuous, one cannot even prove that $(x_n)$ is Cauchy, as the following examples show.

\begin{example}\label{general-fejer-FT-counter}
Let $C:=(0,1]\cup \{ 2\}$ with the metric $d(x,y):=\min \{ |x-y|,1\}.$ Then $C$ is totally bounded and the mapping
\[T:C\to C, \ T(x):=x/2, \text{~if }x\in (0,1],\ T(2):=2\]
 is continuous with  $F:=Fix(T)=\{ 2\}.$ Now let $x_n:=T^n(1)$, for even $n$, and 
$x_n:=1$ for odd $n$.  Then $(x_n)$ has approximate fixed points and 
is Fej\'er monotone w.r.t. $F$ but clearly not Cauchy. 

If we drop the explicit closedness of $F$, we can slightly modify the above example to get a counterexample to the Cauchyness of $(x_n)$ even for compact $C$: just  take 
$C:=[0,1]\cup\{ 2\}$ and define $T(0):=2$. 
\end{example}  

\subsection{Picard iteration for  (firmly) nonexpansive mappings}\label{section-Picard-ne}

Assume that $T$ is nonexpansive. Then, obviously, $T$ is uniformly continuous with  modulus $\omega_T=id_\N$. We consider in the sequel  
the Picard iteration starting from $x\in C$: 
\[x_n:=T^nx.\]

One can see by induction that for all $n,m\in\N$ and $p\in C$, 
\[
d(x_{n+m},p)\le d(x_n,p) + md(p,Tp). 
\]
\details{By induction on $m$. The case $m=0$ is obvious.  $m\Ra m+1$:
\bua
d(x_{n+m+1},p) & = & d(Tx_{n+m},p)\leq d(Tx_{n+m},Tp)+d(Tp,p)\leq d(x_{n+m},p)+d(Tp,p)\\
& \leq & d(x_n,p) + (m+1)d(p,Tp) 
\quad \text{by the induction hypothesis}
\eua
}
As an immediate consequence, we get that $(x_n)$ is Fej\'er monotone, hence, in particular, bounded. In fact, one can easily prove more:

\begin{lemma}\label{lemma-Picard-ne-unif-Fejer}
$(x_n)$ is uniformly Fej\' er monotone w.r.t. $F$ with modulus 
\[\chi(n,m,r)=m(r+1).\]
\end{lemma}
\details{\begin{proof}
Let $r,n,m\in\NN$ and $p\in C$ be such that $d(p,Tp)\leq \frac1{\chi(n,m,r)+1}=\frac1{m(r+1)+1}<\frac1{m(r+1)}$. It follows that for all $l\le m$,
\bua
d(x_{n+l},p) &\le & d(x_n,p) + ld(p,Tp)\le d(x_n,p) + md(p,Tp)\\
&\le &  d(x_n,p) +m\frac1{m(r+1)}\le d(x_n,p) +\frac1{r+1}. 
\eua
\end{proof}
}

Applying Proposition \ref{general-fejer-FT}, we get  

\begin{corollary}\label{general-fejer-FT-Picard-ne}
Let $C$ be a boundedly compact subset of a metric space $X$ and $T:C\to C$ be nonexpansive with $Fix(T)\ne\emptyset$. 
Assume that $(x_n)$  has approximate fixed points. Then $(x_n)$  converges to a fixed point of $T$. 
\end{corollary}

As $(x_n)$ is uniform Fej\'er monotone  w.r.t. $F$ and $F$ is uniformly closed, we can apply our quantitative Theorems \ref{main-step-general-GH-fejer}  and \ref{finitization-general-GH-Fejer}  to get the following:

\begin{theorem}\label{Picard-ne-general-GH-fejer}
Assume that $C$ is totally bounded with II-modulus of total boundedness $\gamma$, $T:C\to C$ is nonexpansive with $Fix(T)\ne\emptyset$ and that $(x_n)$ has approximate fixed points, with $\Phi$ being an approximate fixed point bound.
Then for all $k\in\N$ and all $g:\N\to\N$,
\be
\item There exists $N\le \Sigma(k,g,\Phi,\gamma)$ such that 
\[\forall i,j\in [N,N+g(N)]\left( d(x_i,x_j)\le \frac1{k+1}\right), \]
where $\Sigma(k,g,\Phi,\gamma)=\Sigma_0(\gamma(4k+3),k,g,\Phi)$, with $\Sigma_0(0,k,g,\Phi)=0$ and $$\Sigma_0(n+1,k,g,\Phi) =
\Phi\left((4k+4)g^M(\Sigma_0(n,k,g,\Phi))\right).$$
\item There exists $N\le \tilde{\Sigma}(k,g,\Phi,\gamma)$ such that 
\[\forall i,j\in [N,N+g(N)]\
 \left( d(x_i,x_j)\le \frac1{k+1} \text{~and~} d(x_i,Tx_i)\le \frac1{k+1}  \right), \]
where $\tilde{\Sigma}(k,g,\Phi,\gamma)=\tilde{\Sigma}_0(\gamma(8k+7),k,g,\Phi)$, with $\tilde{\Sigma}_0(0,k,g,\Phi)=0$ and 
$$\tilde{\Sigma}_0(n+1,k,g,\Phi)=\Phi\left(\max\left\{2k+1, (8k+8)g^M(\tilde{\Sigma}_0(n,k,g,\Phi))\right\}\right).$$
\ee
\end{theorem} 
\begin{proof}
\be
\item With $\Psi,\Psi_0$ as in Theorem \ref{main-step-general-GH-fejer}, $\alpha_G=\beta_H=id_\N$ and $\chi$ as in Lemma \ref{lemma-Picard-ne-unif-Fejer}, 
define  $\Sigma(k,g,\Phi,\gamma)=\Psi(k,g^M,\Phi,\chi,\alpha_G,\beta_H,\gamma)$  and $\Sigma_0(l,k,g,\Phi) = \Psi_0(l,k_0,g^M,\Phi,\chi,\beta_H)$.

\details{
We have that 
\bua
P&=& \gamma\left( \alpha_G\left(2\beta_H(2k+1)+1\right)\right) =\gamma(4k+3)\\
\chi_{g^M}(n,r)&=& \chi(n,g^M(n),r)=g^M(n)(r+1)\\
\chi^M_{g^M}(n,r)&=& \max\{ \chi_{g^M}(i,r) \mid  i\le n\}=\max\{g^M(i)(r+1) \mid i\le n\}\\
&=& g^M(n)(r+1)=\chi_{g^M}(n,r).
\eua
Hence,
\bua
\Sigma(k,g,\Phi,\gamma)=\Sigma_0(P,k,g,\Phi)=\Sigma_0(\gamma(4k+3),k,g,\Phi),
\eua
$\Sigma_0(0,k,g,\Phi)=0$ and
\bua
\Sigma_0(n+1,k,g,\Phi) &=& \Psi_0(n+1,k,g^M,\Phi,\chi,\beta_H) \\
&=& \Phi\left(\chi^M_{g^M}\left(\Psi_0(n,k,g^M,\Phi,\chi,\beta_H),2\beta_H(2k+1)+1\right)\right)\\
&=& \Phi\left(\chi^M_{g^M}\left(\Sigma_0(n,k,g,\Phi),4k+3\right)\right)\\
&=& \Phi\left((4k+4)g^M(\Sigma_0(n,k,g,\Phi))\right).
\eua
}

\item  Apply Theorem \ref{finitization-general-GH-Fejer} for $g^M$, using that $\omega_F(k)=4k+3$ and $\delta_F(k)=2k+1$, by Lemma \ref{FixT-xn-basic}.\eqref{Tuc-Fuc}.
It follows that  $k_0=2k+1$ and $(\chi_k)^M_{g^M}(n,r)= \max\{2k+1, g^M(n)(r+1)\}$.

\details{Using the fact that  $\delta_F(k)=2k+1$ and $\omega_F(k)=\max\{4k+3,\omega_T(4k+3)\}=4k+3$ , 
we get that
\bua
k_0&=&\max\left\{ k,\left\lceil\frac{\omega_F(k)-1}2\right\rceil \right\}=\max\{k,2k+1\}=2k+1\\
\chi_{k}(n,m,r)&:=& \chi_{k,\delta_F}(n,m,r)=  \max\{\delta_F(k),\chi(n,m,r)\}=\max\{2k+1, m(r+1)\}\\
\eua
It follows that 
\bua
P &=& \gamma\left( \alpha_G\left(2\beta_H(2k_0+1)+1\right)\right)=\gamma(4k_0+3)=\gamma(8k+7)\\
(\chi_k)_{g^M}(n,r)&=&\chi_k(n,g^M(n),r)=\max\{2k+1, g^M(n)(r+1)\}\\
(\chi_k)^M_{g^M}(n,r)&=&\max\{ (\chi_k)_{g^M}(i,r): i\le n\}=\max\{\max\{2k+1,g^M(i)(r+1): i\le n\}\\
&=& \max\{2k+1, g^M(n)(r+1)\}=(\chi_k)_{g^M}(n,r).
\eua
Hence,
\bua
\tilde{\Sigma}(k,g,\Phi,\gamma)=\tilde{\Sigma}_0(P,k,g,\Phi)=\tilde{\Sigma}_0(\gamma(8k+7),k,g,\Phi),
\eua
$\tilde{\Sigma}_0(0,k,g,\Phi)=0$ and
\bua
\tilde{\Sigma}_0(n+1,k,g,\Phi) &=& \Psi_0(n+1,k_0,g^M,\Phi,\chi_k,\beta_H) \\
&=& \Phi\left((\chi_k)^M_{g^M}\left(\Psi_0(n,k_0,g^M,\Phi,\chi_k,\beta_H),2\beta_H(2k_0+1)+1\right)\right)\\
&=& \Phi\left((\chi_k)^M_{g^M}\left(\tilde{\Sigma}_0(n,k,g,\Phi),8k+7\right)\right)\\
&=& \Phi\left(\max\{2k+1, (8k+8)g^M(\tilde{\Sigma}_0(n,k,g,\Phi))\}\right)
\eua
}
 \ee
\end{proof}

If, moreover,  $(x_n)$ is asymptotic regular and  we can compute a rate of asymptotic regularity $\Phi^{++}$,  then we can also apply 
Corollary \ref{cor-self-maj}.

\begin{theorem}\label{Picard-ne-general-GH-fejer-rate-as-reg}
Assume that $C$ is totally bounded with II-modulus of total boundedness $\gamma$, $T:C\to C$ is nonexpansive with $Fix(T)\ne\emptyset$ and that 
$(x_n)$ is asymptotic regular with $\Phi^{++}$ being a rate of asymptotic regularity.
Then for all $k\in\N$ and all $g:\N\to\N$, there exists $N\le \Theta(k,g,\Phi^{++},\gamma)$ such that 
\[ \forall i,j\in [N,N+g(N)]\,\forall m\ge N\,\left( d(x_i,x_j)\le \frac1{k+1} \text{~and~}  d(x_m,Tx_m)\le \frac1{k+1} \right),\]  
where 
\[\Theta(k,g,\Phi^{++},\gamma) =\Theta_0(\gamma^M(4k+3),k,g,\Phi^{++})+K,\]
with $K=(\Phi^{++})^M(k)$, $\Theta_0(0,k,g,\Phi^{++})=0$ and $$
\Theta_0(n+1,k,g,\Phi^{++})=(\Phi^{++})^M\left((g^M(\Theta_0(n,k,g,\Phi^{++})+K)+K)(4k+4) \right).$$
\end{theorem} 
\begin{proof}
Define  $\Theta(k,g,\Phi^{++},\gamma) = \tilde{\Omega}(k,g,\Psi^+,\Phi^{++})$, with $\tilde{\Omega}$ as in Corollary \ref{cor-self-maj}, $\Phi=(\Phi^{++})^M$ and $\Psi^+(k,g)=\Sigma(k,g,\Phi,\gamma)$,  where $\Sigma$ is defined in Theorem \ref{Picard-ne-general-GH-fejer}.(i).
   
\details{
We have that 
\bua
\Theta(k,g,\Phi^{++},\gamma) &=& \tilde{\Omega}(k,g,\Psi^+,\Phi^{++}) = \tilde{\Omega}_{k,g}(\Psi^+,(\Phi^{++})^M)\\
&=& \Psi^+(k, h)+K.
\eua
where 
\bua 
h(n) &=& g_{K}(n)=g^M(n+K)+K.
\eua
Let us remark that $\alpha_G^M=\alpha_G=\beta_H^M=\beta_H=id_\N$ and $\chi^M=\chi$. Furthermore, 
\bua
\chi_{h}(n,r)&=& \chi(n,h(n),r)=h(n)(r+1)= (g^M(n+K)+K)(r+1) \\
\chi^M_{h}(n,r)&=& \max\{ \chi_{h}(i,l): i\le n\}=l+\max\{(g^M(i+K)+K)(r+1) : i\le n\}\\
&=&(g^M(n+K)+K)(r+1)=\chi_{h}(n,r).
\eua

If we denote
$\Theta_0(n,k,g,\Phi^{++})= \Psi_0(n, k,h,(\Phi^{++})^M,\chi,id_\N)$, we get that 

\bua
\Psi^+(k, h) &=& \Psi(k,h,(\Phi^{++})^M,\chi^M,\alpha^M_G,\beta^M_H,\gamma^M)= \Psi(k,h,(\Phi^{++})^M,\chi,id_\N,id_\N,\gamma^M)\\
&=&
\Psi_0(\gamma^M(k+2),k,h,(\Phi^{++})^M,\chi,id_\N)=\Theta_0(\gamma^M(4k+3),k,g,\Phi^{++}), 
\eua
where 
$\Theta_0(0,k,g,\Phi^{++})=0$ and
\bua
\Theta_0(n+1,k,g,\Phi^{++}) &=& \Psi_0(n+1,k,h,(\Phi^{++})^M,\chi,id_\N)\\
&=& (\Phi^{++})^M\left(\chi^M_{h}\left(\Psi_0(n,k,h,(\Phi^{++})^M,\chi,id_\N),4k+3\right)\right)\\
&=& (\Phi^{++})^M\left(\chi^M_{h}\left(\Theta_0(n,k,g,\Phi^{++}),4k+3\right)\right)\\
&=& (\Phi^{++})^M\left((g^M(\Theta_0(n,k,g,\Phi^{++})+K)+K)(4k+4)\right).
\eua
}
\end{proof}

We recall that a $W$-hyperbolic space \cite{Koh05a} is a metric space $X$ endowed with a convexity mapping $W:X\times X\times [0,1]\to X$ satisfying
\begin{eqnarray*}
(W1) & d(z,W(x,y,\lambda))\le (1-\lambda)d(z,x)+\lambda d(z,y),\\
(W2) & d(W(x,y,\lambda),W(x,y,\tilde{\lambda}))=|\lambda-\tilde{\lambda}|\cdot 
d(x,y),\\
(W3) & W(x,y,\lambda)=W(y,x,1-\lambda),\\
(W4) & \,\,\,d(W(x,z,\lambda),W(y,w,\lambda)) \le (1-\lambda)d(x,y)+\lambda
d(z,w).
\end {eqnarray*}
for all $x,y,z\in X$ and all $\lambda,\tilde{\lambda}\in [0,1]$.  We use in the sequel the notation $(1-\lambda)x+\lambda y$ for $W(x,y,\lambda)$.  

Following \cite{GoeRei84}, one can define in the setting of $W$-hyperbolic spaces a notion of uniform convexity.  A $W$-hyperbolic space $X$ is {\em uniformly convex with modulus} 
$\eta:(0,\infty)\times(0,2]\rightarrow (0,1]$ if for any $r<0,\eps\in (0,2]$ and all $a,x,y\in X$,  
\[ d(x,a)\le r, \, d(y,a)\le r, \text{ and } d(x,y)\ge\varepsilon r \quad \text{imply} \quad d\left(\frac12x+\frac12y,a\right)\le (1-\eta(r,\eps))r.\]
A modulus $\eta$ is said to be {\em monotone} if it is nonincreasing in the first argument. 
Uniformly convex $W$-hyperbolic spaces with a monotone modulus $\eta$ are called $UCW$-hyperbolic spaces in \cite{Leu10}.  One can easily see that CAT(0) spaces \cite{BriHae99} are $UCW$-hyperbolic spaces with modulus $\eps^2/8$. We refer to  \cite{Leu10,Leu07,KohLeu10} for  properties of $UCW$-hyperbolic spaces.

\mbox{}

A very important class of nonexpansive mappings are the firmly nonexpansive ones. They are central  in convex optimization  because of  the correspondence with maximal monotone operators due to Minty \cite{Min62}. We refer to \cite{BauMofWan12} for a systematic analysis of this correspondence. Firmly nonexpansive mappings were introduced by Browder \cite{Bro67} in Hilbert spaces and by Bruck \cite{Bru73} in Banach spaces, but they are also studied in the Hilbert ball \cite{GoeRei84,KopRei09} or in different classes of geodesic spaces \cite{ReiSha87,ReiSha90,AriLeuLop14,Nic13}.

Let  $C\se X$ be a nonempty subset of a $W$-hyperbolic space $X$.  A mapping 
$T:C\to C$ is  {\em $\lambda$-firmly nonexpansive} (where $\lambda\in (0,1)$) if for all $x,y\in C$,
\[
d(Tx,Ty)\leq d((1-\lambda)x+\lambda Tx,(1-\lambda)y+\lambda Ty)\leq d(x,y).
\]

Using proof mining methods, effective uniform rates of asymptotic regularity for the Picard iteration were obtained for $UCW$-hyperbolic spaces in \cite{AriLeuLop14} and for $W$-hyperbolic spaces in \cite{Nic13}.   For  a CAT(0) space $X$, $C\se X$ a bounded subset, one gets, as an immediate consequence of \cite[Theorem 7.1]{AriLeuLop14}\footnote{Correction to \cite{AriLeuLop14}: In Corollary 7.4 should be `$(b+1)^2$' instead of `$(b+1)$' in the definition of $\Phi(\varepsilon,\lambda,b)$.} the following rate of asymptotic regularity for the Picard iteration of a $\lambda$-firmly nonexpansive mapping $T:C\to C$:  
\beq
\Phi^{++}(k,b,\lambda):=\left\lceil\frac{8(b+1)^2}{\lambda\,(1-\lambda)}\right\rceil (k+1)^2, \label{rate-as-reg-Picard-fne}
\eeq
where  $b>0$ is an upper bound on the diameter of $C$.

We can, thus, apply Theorem \ref{Picard-ne-general-GH-fejer-rate-as-reg} and remark that $(\Phi^{++})^M=\Phi^{++}$ to obtain

\begin{corollary}\label{Picard-fne-GH-fejer}
Assume that $X$ is a CAT(0) space, $C\se X$ a totally bounded subset and $T:C\to C$ is a 
$\lambda$-firmly nonexpansive mapping with $Fix(T)\ne\emptyset$. 
Let $\gamma$ be a II-modulus of total boundedness of $C$ and $b>0$ be an upper bound on 
the diameter of $C$. Then for all $k\in\N$ and all $g:\N\to\N$,
there exists $N\le \Theta(k,g,\gamma,b,\lambda)$ such that 
\[ \forall i,j\in [N,N+g(N)]\,\forall m\ge N\,\left( d(x_i,x_j)\le \frac1{k+1} \text{~and~}  d(x_m,Tx_m) 
\le \frac1{k+1} \right),\]
where $\Theta\left(k,g,\gamma,b,\lambda\right):=\Theta_0(\gamma^M(4k+3),k,g,b,\lambda)$, 
with 
\bua 
c=\left\lceil\frac{8(b+1)^2}{\lambda\,(1-\lambda)}\right\rceil, \quad  K=c\cdot (k+1)^2, \,
 \Theta_0(0,k,g,b,\lambda)=0 \text{~and} \\
 \Theta_0(n+1,k,g,b,\lambda) = c\left((g^M(\Theta_0(n,k,g,b,\lambda)+K)+K)(4k+4)\right)^2.
 \eua
\end{corollary}
\details{
\bua
\Theta_0(n+1,k,g,b,\lambda) &=& (\Phi^{++})^M\left((g^M(\Theta_0(n,k,g,b,\lambda)+K)+K)(4k+4) \right)\\
&=& c\left((g^M(\Theta_0(n,k,g,b,\lambda)+K)+K)(4k+4)\right)^2
\eua
}

\subsection{Ishikawa iteration for nonexpansive mappings}

Assume that $X$ is a $W$-hyperbolic space, $C\se X$ is convex and $T:C\to C$ is nonexpansive. The {\em Ishikawa iteration} 
starting with $x\in C$ is defined as follows:
\beq
x_0:=x, \quad x_{n+1}:=(1-\lambda_n)x_n+ \lambda_nT((1-s_n)x_n+s_nTx_n), \label{def-Ishikawa}
\eeq
where $(\lambda_n),(s_n)$ are sequences in $[0,1]$.  This iteration was introduced in \cite{Ish74} in the setting of Hilbert spaces  and 
it is a generalization of the well-known Mann iteration \cite{Man53,Gro72}, 
which can be obtained as a special case of \eqref{def-Ishikawa} by taking $s_n=0$ for all $n\in\N$.

\begin{lemma}\label{lemma-Ishikawa-ne-basic}
\be
\item For all $n,m\in\N$ and all $p\in C$,
\bea
d(x_{n+1},p)& \le &  d(x_n,p)+2\lambda_n d(p,Tp) \label{Ish-ineq-1}\\
d(x_{n+m},p) &\le & d(x_n,p) + 2md(p,Tp). \label{ineq-quant-Ishikawa-ne}
\eea
\item $(x_n)$ is uniformly Fej\' er monotone w.r.t. $F$ with modulus 
\[\chi(n,m,r)=2m(r+1)\]
\ee
\end{lemma}
\begin{proof}
\be
\item 
\eqref{Ish-ineq-1} is proved in \cite[Lemma 4.3, (11)]{Leu14}.
We get \eqref{ineq-quant-Ishikawa-ne} by an easy induction.
\details{The case $m=1$ follows from \eqref{Ish-ineq-1}.
$m\Ra m+1$: We have that 
\bua
d(x_{n+m+1},p)&\leq & d(x_{n+m},p)+2\lambda_{n+m} d(p,Tp)\leq d(x_{n+m},p)+2d(p,Tp)\\
&\leq & d(x_n,p)+2(m+1)d(p,Tp) \quad \text{ by the induction hypothesis}.
\eua
}
\item follows easily from \eqref{ineq-quant-Ishikawa-ne}.
\details{
Let $r,n,m\in\NN$ and $p\in C$ be such that $d(p,Tp)\leq \frac1{\chi(n,m,r)}=\frac1{2m(r+1)+1}<\frac1{2m(r+1)}$. It follows that for all $l\le m$,
\bua
d(x_{n+l},p) &\le & d(x_n,p) + 2ld(p,Tp)\le d(x_n,p) + 2md(p,Tp)\\
&\le &  d(x_n,p) + 2m\cdot  \frac1{2m(r+1)} < d(x_n,p) + \frac1{r+1}. 
\eua
}
\ee
\end{proof}

As in the case of the Picard iteration of a nonexpansive mapping, we can apply  Proposition \ref{general-fejer-FT} to get that for boundedly compact $C$ and $T:C\to C$ nonexpansive with $Fix(T)\ne\emptyset$, the fact that Ishikawa iteration $(x_n)$  has approximate fixed points implies the convergence of $(x_n)$ to a fixed point of $T$. 
Explicit approximate fixed point bounds and rates of asymptotic regularity w.r.t. $F$ are computed in \cite{Leu10} for closed convex subsets $C$ of $UCW$-hyperbolic spaces $X$.

 We shall consider in the following only the setting of CAT(0) spaces. We assume that 
\be
\item $\ds\sum_{n=0}^\infty\lambda_n(1-\lambda_n)$ is divergent with $\theta :\N\to\N$ being a nondecreasing rate of divergence, i.e. satisfying $\ds\sum_{k=0}^{\theta(n)}\lambda_k(1-\lambda_k)\geq n$ for all $n$.
\item $\limsup_n s_n<1$ and $L,N_0\in\N$ are such that $\ds s_n\leq 1-\frac1L$ for all $n\geq N_0$.
\ee
Then, as a consequence of \cite[Corollary 4.6]{Leu10} and the proof of \cite[Remark 4.8]{Leu10}, we get the following approximate fixed point bound for $(x_n)$.

\begin{proposition}\label{CAT0-Ishkawa-lambdan}
Let $X$ be a $CAT(0)$ space, $C\se X$  a bounded convex closed subset with diameter $d_C$ and $T:C\rightarrow C$ nonexpansive. Then
\beq
\forall k\in\N\, \exists N\le \Phi(k,b,\theta,L,N_0) \left( d(x_N,Tx_N) \le \frac1{k+1} \right),
\eeq
where $\Phi(k,b,\theta,L,N_0)=\theta\left(4(k+1)^2L^2\lceil b(b+1)\rceil+N_0\right)$, 
with $b>0$ being an upper bound on the diameter of $C$.
\end{proposition}
\details{In general, let $(X,d,W)$ be $UCW$-hyperbolic space with monotone modulus of uniform convexity $\eta$.  
By \cite[Corollary 4.6]{Leu10} we have that 
\beq
\forall \eps>0\, \exists N\le \tilde{\Phi}(\varepsilon,\eta,b,\theta,L,N_0)\bigg( d(x_N,Tx_N)<\varepsilon\bigg),
\eeq
where $\Phi(\eps,\eta,b,\theta,L,N_0):=\Psi(\eps,0,\eta,b,\theta,L,N_0)$, with $\Psi$ defined in \cite[Proposition 4.5]{Leu10}:
\beq
\Psi(\varepsilon,0,\eta,b,\theta,L,N_0):=h\left(\frac{\eps}L, N_0,\eta,b,\theta\right),
\eeq
where $h$ is defined in \cite[Proposition 4.3]{Leu10}:
\[h(\varepsilon,l,\eta,b,\theta):=\theta\left(\left\lceil\frac{b+1}{\varepsilon\cdot\eta\left(b,\displaystyle\frac{\varepsilon}{b}\right)}\right\rceil+l\right)\]
Assume that $X$ is CAT(0), with modulus $\eta(r,\eps)=\frac{\eps^2}8$.
By the proof of \cite[Remark 4.8]{Leu10},  as $\eta(r,\eps)=\eps\tilde{\eta}(r,\eps)$, with $\tilde{\eta}(r,\eps)=\frac{\eps}8$, we can take 
\bua
h(\varepsilon,l,\eta,b,\theta) &= & \theta\left(\left\lceil\frac{b+1}{2\varepsilon\cdot\tilde{\eta}\left(b,\displaystyle\frac{\varepsilon}{b}\right)}\right\rceil+l\right)=\theta\left(\left\lceil\frac{8b(b+1)}{2\varepsilon^2}\right\rceil+l\right)=\theta\left(\left\lceil\frac{4b(b+1)}{\varepsilon^2}\right\rceil+l\right).
\eua         
Thus,  
\bua
\Phi(\eps,b,\theta,L,N_0)&=& \Psi(\eps,0,b,\theta,L,N_0)= h\left(\frac{\eps}L, N_0,b,\theta\right) =\theta\left(\left\lceil\frac{4L^2b(b+1)}{\varepsilon^2}\right\rceil+N_0\right).
\eua
By letting $\eps=\frac1{k+1}$, we get that 
\bua
\Phi\left(\frac1{k+1},b,\theta,L,N_0\right)&=& \theta\left(\left\lceil 4(k+1)^2L^2b(b+1)\right\rceil+N_0\right)\leq \theta\left(4(k+1)^2L^2\lceil b(b+1)\rceil+N_0\right)\\
&=& \Phi(k,b,\theta,L,N_0),
\eua
since $\theta$ is nondecreasing and $\lceil am\rceil \leq \lceil a\rceil m$ for every $a>0$, $m\in\N$. 
}

For the particular case $\lambda_n=\lambda$, one can take $\ds \theta(n)=n\left\lceil\frac{1}{\lambda(1-\lambda)}\right\rceil$,
hence the approximate fixed point bound $\Phi$ becomes
\[\Phi(k,b,L,N_0)=\left\lceil\frac{1}{\lambda(1-\lambda)}\right\rceil\left(4(k+1)^2L^2\lceil b(b+1)\rceil+N_0\right).\]

Finally, we can apply  Theorem \ref{finitization-general-GH-Fejer} to get for $C$ totally bounded with II-modulus of total boundedness $\gamma$ a result similar with Theorem \ref{Picard-ne-general-GH-fejer}.(ii),  providing us a  functional 
$\tilde{\Sigma}:=\tilde{\Sigma}(k,g,\Phi,\gamma)$  with the property that for all $k\in\N$ and all $g:\N\to\N$ there exists $N\le \tilde{\Sigma}$ such that 
\[ \forall i,j\in\! [N,N+g(N)]\,\,\left( d(x_i,x_j)\le \frac1{k+1} \text{~and~}  d(x_i,Tx_i) 
\le \frac1{k+1}\right).\]

\subsection{Mann iteration for strictly pseudo-contractive mappings}

Assume that $X$ is a real Hilbert space, $C\se X$  is a nonempty bounded closed convex subset with finite diameter $d_C$ and $0\leq \kappa < 1$. 

A mapping $T:C\to C$ is a \emph{$\kappa$-strict pseudo-contraction}  if for all $x,y\in C$,
\beq 
 \|Tx - Ty\|^2 \leq \|x-y\|^2 + \kappa \|x - Tx- \left(y -T y \right)\|^2. \label{def-k-strict-pseudo-contraction}
\eeq
This definition was given by Browder and Petryshyn in \cite{BroPet67}, where they also proved that under the above hypothesis 
$F=Fix(T)\ne\emptyset$.  Obviously, a mapping $T$ is nonexpansive if and only if $T$ is a $0$-strict pseudo-contraction.

In the following, $T$ is a $\kappa$-strict pseudo-contraction. Then  $T$ is Lipschitz continuous with Lipschitz constant $\ds L= \frac{1+\kappa}{1-\kappa}$ (see \cite{MarXu07}), hence  $T$ is uniformly continuous with modulus  $\omega_T(k)=L(k+1)$.  
By  Lemma \ref{FixT-xn-basic}.\eqref{Tuc-Fuc}, it follows that $F$ is uniformly closed with moduli 
\[\omega_F(k)= L(4k+4)\text{~and~} \delta_F(k)=2k+1.\] 

\details{If $d(x,y)\leq \frac1{\omega_T(k)+1}=\frac1{L(k+1)+1}< \frac1{L(k+1)}$, then $ d(Tx,Ty) \leq L d(x,y)\le \frac1{k+1}$. Furthermore,
\bua
 \omega_F(k)=\max\{4k+3,\omega_T(4k+3)\}=\max\{4k+3, L(4k+4)\}=L(4k+4)
\eua
since $L\geq 1$. 
}

We consider the Mann iteration associated to $T$ which, as we remarked above, is defined  by 
\beq
x_0:=x, \quad x_{n+1}:=(1-\lambda_n)x_n+\lambda_nTx_n, \label{def-Mann-ne}
\eeq
where $(\lambda_n)$ is  sequences in $(0,1)$.

\begin{lemma}\label{lemma-KM-k-basic}
Assume that $(\lambda_n)$  is a sequence in $(\kappa,1)$ and  let $b\geq d_C$. Then 
\be
\item For all $n,m\in\N$ and all $p\in C$,
\bea
\|x_{n+1} - p\|^2 &\leq &  \|x_n-p\|^2+2b(n+3)\|p-Tp\| \label{KM-k-ineq-1}\\
\|x_{n+m} - p\|^2 &\leq &   \|x_n-p\|^2 + mb(2n+m+5)\|p-Tp\|. \label{KM-k-ineq-m}
\eea
\item $(x_n)$ is uniformly $(G,H)$-Fej\' er monotone w.r.t. $F$ with modulus 
\[\chi(n,m,r)=m(2n+m+5)(r+1)\lceil b\rceil,\]
where $G(a)=H(a)=a^2$ with  $G$-modulus $\alpha_G(k)=\left\lceil \sqrt{k}\right\rceil$ and $H$-modulus 
$\beta_H(k)=k^2$.
\ee
\end{lemma}
\begin{proof}
\be
\item  \eqref{KM-k-ineq-1} follows from \cite[Lemma 3.4.(ii)]{IvaLeu14}. We prove that 
\[\|x_{n+m} - p\|^2 \leq   \|x_n-p\|^2 +2b\sum_{k=0}^{m-1}(n+k+3)\|p-Tp\|\]
by induction on $m$.
\details{The case $m=1$ follows from \eqref{KM-k-ineq-1}.
$m\Ra m+1$: We have that 
\bua
\|x_{n+m+1}-p\|^2 &\leq & \|x_{n+m}-p\|^2 +2b(n+m+3)\|p-Tp\| \\
&\leq &  \|x_n-p\|^2 +2b\sum_{k=0}^{m-1}(n+k+3)\|p-Tp\|\\
&& + 2(n+m+3)b\|p-Tp\|  \quad \text{ by the induction hypothesis}\\
&=&  \|x_n-p\|^2 +2b\sum_{k=0}^{m}(n+k+3)\|p-Tp\|\\
&=&  \|x_n-p\|^2 +2b\sum_{k=1}^{m}(n+k+3)\|p-Tp\|.
\eua
Remark that 
\bua 
\sum_{k=0}^{m-1}(n+k+3) &=& m(n+3)+(1+\ldots +m-1)=m(n+3)+\frac{m(m-1)}2\\
&=& \frac{m(2n+6+m-1)}2=\frac{m(2n+m+5)}2.
\eua
}
\item  Apply \eqref{KM-k-ineq-m}.
\details{
Let $r,n,m\in\N$ and $p\in C$ be such that $\|p-Tp\|\leq \frac1{\chi(n,m,r)+1}=\frac1{m(2n+m+5)(r+1)\lceil b\rceil+1}$. The case $m=0$ is trivial, hence we can take $m\geq 1$.
It follows that for all $l\le m$,
\bua
\|x_{n+l}-p\|^2  &\le &   \|x_n-p\|^2 + mb(2n+m+5)\|p-Tp\|\\
&\le &  \|x_n-p\|^2 +mb(2n+m+5)\frac1{m(2n+m+5)(r+1)\lceil b\rceil+1}\\
&< & mb(2n+m+5)\frac1{m(2n+m+5)(r+1)\lceil b\rceil}\\
&\leq &   \|x_n-p\|^2+\frac1{r+1}.
\eua}
Assume that $a\leq \frac1{\alpha_G(k)+1}\leq\frac1{\sqrt{k}+1}$. Then $G(a)=a^2\leq \frac1{k+1+2\sqrt{k}}\leq \frac1{k+1}$.  
Assume that $H(a)=a^2\leq \frac1{\beta_H(k)+1}\leq\frac1{(k+1)^2}$. Then $a\leq \frac1{k+1}$.  
\ee
\end{proof}

Effective rates of asymptotic regularity for the Mann iteration $(x_n)$ are computed in \cite{IvaLeu14}: if $(\lambda_n)$  is a sequence in $(\kappa,1)$  satisfying 
$\ds \sum_{n=0}^\infty (\lambda_n -\kappa)(1-\lambda_n) = \infty$ with rate of divergence $\theta:\N\to\N$, then 
\beq
\Phi^{++}(k,b,\theta)=\theta\left(\lceil b^2\rceil (k+1)^2 \right).
\eeq
is  a rate of asymptotic regularity for $(x_n)$.  Thus, 
$$(\Phi^{++})^M(k,b,\theta)=\theta^M\left(\lceil b^2\rceil (k+1)^2\right).$$  
If $\lambda_n=\lambda$, one gets the following rate of asymptotic regularity for the Krasnoselskii iteration:  
\beq
\Phi^{++}(k,b,\kappa, \lambda)=\left\lceil\frac{b^2}{(\lambda-\kappa)(1-\lambda)}\right\rceil (k+1)^2.
\eeq

Thus, Corollary  \ref{cor-self-maj}  can be applied now to obtain rates of metastability for the Mann iteration, in the case when $C$ is totally bounded.

\subsection{Mann iteration for mappings satisfying condition $(E)$}
\label{section-E}

Assume that $X$ is a $W$-hyperbolic space and $C \subseteq X$ is nonempty and convex. Let $T : C \to C$ and $\mu \ge 1$. The mapping $T$ satisfies {\it condition $(E_\mu)$} if for all $x, y \in C$,
\[d(x,Ty) \le \mu d(Tx,x) + d(x,y).\]
$T$ is said to satisfy {\it condition $(E)$} if it satisfies $(E_\mu)$ for some $\mu \ge 1$. This condition was introduced in \cite{GarLloSuz11} as a generalization of condition $(C)$ studied in \cite{Suz08}. Note that condition $(C)$ is a generalization of nonexpansivity and implies $(E_3)$. 

We suppose next that $T$ is a mapping satisfying condition $(E_\mu)$ with $\mu \ge 1$. Then $F$ is uniformly closed with moduli 
\[\delta_F(k)=2\mu(k+1)-1 \text{~and~} \omega_F(k)=4k+3.\] 
Indeed, for $k \in \N$, $p,q \in X$ with $d(q,Tq) \le 1/(2\mu(k+1))$ and $d(p,q) \le 1/(4(k+1))$ we have that
\[d(p,Tp) \le d(p,q) + d(q,Tp) \le 2d(p,q) + \mu d(q,Tq) \le  \frac{1}{k+1}.\]

Let $(x_n)$ be the Mann iteration starting with $x \in C$ defined as follows: 
\[x_0 =x, \quad x_{n+1} =(1-\lambda_n)x_n+\lambda_n T(x_n),\]
where $(\lambda_n) \subseteq [1/L,1-1/L]$ for some $L \ge 2$.

\begin{lemma}\label{lemma-KM-E-basic}
\be
\item For all $n,m\in\N$ and all $p\in C$,
\bea
d(x_{n+m},p) &\le & d(x_n,p) + \mu l (1-1/L) d(p,Tp). \label{KM-cond(E)-ineq}
\eea
\item $(x_n)$ is uniformly Fej\' er monotone w.r.t. $F$ with modulus 
\[\chi(n,m,r) = \mu m(1-1/L)(r + 1).\]
\ee
\end{lemma}
\begin{proof}
\be
\item is proved by induction. When $m=0$ this is clear. Suppose 
\[d(x_{n+m},p) \le d(x_n,p) + \mu m (1-1/L) d(p,Tp).\] 
Then
\bua
d(x_{n+m+1},p) & = & d((1-\lambda_{n+m})x_{n+m} + \lambda_{n+m} Tx_{n+m}, p)\\
& \le & (1-\lambda_{n+m})d(x_{n+m},p) + \lambda_{n+m} d(Tx_{n+m}, p)\\
& \le & (1-\lambda_{n+m})d(x_{n+m},p) + \lambda_{n+m} (\mu d(p,Tp) + d(p,x_{n+m}))\\
& \le & d(x_{n+m},p) + \mu (1-1/L) d(p,Tp) \\
& \le & d(x_n,p) + \mu (l+1) (1-1/L) d(p,Tp).
\eua
\item follows easily from \eqref{KM-cond(E)-ineq}.
\details{
Let $r,n,m \in \mathbb{N}$, $p \in X$ with $d(p,Tp) \le 2^{-\chi(n,r,m)}$. By \eqref{KM-cond(E)-ineq} we have that for all $l \le m$,
\[d(x_{n+l},p) \le d(x_n,p) + \mu m (1-1/L)2^{-r- \lceil \log_2 (\mu m (1-1/L) + 1)\rceil} \le d(x_n,p) + 2^{-r}.\]
}
\ee
\end{proof}

We compute next a rate of metastability for the asymptotic regularity of $(x_n)$ w.r.t. $F$ in the setting of $UCW$-hyperbolic spaces.

\begin{lemma}\label{lemma-unif-cv-cond(E)}
Let $(X,d,W)$ be a $UCW$-hyperbolic space with a monotone modulus of uniform convexity $\eta$. Let $x,p\in C$, $n \in \mathbb{N}$ and $\alpha, \beta, \delta, \nu > 0$ such that
\[d(p,Tp) < \nu \le \delta, \quad \alpha \le d(x_n,p)\le \beta, \quad \alpha \le d(x_n,Tx_n).\]
Then
\begin{equation} \label{KM-cond(E)-UCW}
d(x_{n+1},p) < d(x_n,p) + \mu\nu - 2\alpha L^{-2}\eta\left(\mu\delta+\beta, \frac{\alpha}{\mu\delta+\beta}\right).
\end{equation}
If $\eta(r,\varepsilon)\ge \varepsilon\cdot \tilde{\eta}(r,\varepsilon)$ with 
$\tilde{\eta}$ increasing w.r.t. $\varepsilon,$ then one can replace $\eta$ 
by $\tilde{\eta}$ in (\ref{KM-cond(E)-UCW}).
\end{lemma}
\begin{proof}
Let $r_n = \mu d(p,Tp) + d(p,x_n) < \mu \delta + \beta$. Since $d(x_n,p) \le r_n$, $d(Tx_n,p) \le \mu d(p,Tp) + d(p,x_n) = r_n$ and $d(x_n,Tx_n) \ge \alpha > \frac{\alpha}{\mu \delta + \beta}r_n$, by uniform convexity, it follows that
\bua
d(x_{n+1},p) & \le & \left(1-2\lambda_n(1-\lambda_n)\eta\left(r_n, \frac{\alpha}{\mu\delta+\beta}\right)\right)r_n \quad \text{by  \cite[Lemma 7]{Leu07} }\\
& \le & \left(1-2\lambda_n(1-\lambda_n)\eta\left(\mu\delta+\beta, \frac{\alpha}{\mu\delta+\beta}\right)\right)r_n\\
& \le & r_n - 2r_nL^{-2}\eta\left(\mu\delta+\beta, \frac{\alpha}{\mu\delta+\beta}\right)\\
& \le & d(p,x_n) + \mu d(p,Tp) - 2\alpha L^{-2}\eta\left(\mu\delta+\beta, \frac{\alpha}{\mu\delta+\beta}\right)\\
& < & d(p,x_n) + \mu \nu - 2\alpha L^{-2}\eta\left(\mu\delta+\beta, \frac{\alpha}{\mu\delta+\beta}\right).
\eua
The additional claim follows using $\alpha/r_n$ instead of $\alpha/(\mu\delta
+\beta):$
\bua
d(x_{n+1},p) & 
 \le & r_n - 2r_nL^{-2}\eta\left(\mu\delta+\beta, \frac{\alpha}{r_n}\right)\\
& \le & r_n - 2\alpha L^{-2}\tilde{\eta}\left(\mu\delta+\beta, \frac{\alpha}{r_n}\right)\\
& \le & r_n - 2\alpha L^{-2}\tilde{\eta}\left(\mu\delta+\beta, \frac{\alpha}{\mu\delta +\beta}\right)\\
& < & d(p,x_n) + \mu \nu - 2\alpha L^{-2}\tilde{\eta}\left(\mu\delta+\beta, \frac{\alpha}{\mu\delta+\beta}\right).
\eua
\end{proof}

\begin{theorem}
Let $(X,d,W)$ be a $UCW$-hyperbolic space with a monotone modulus of uniform convexity $\eta$. Let $x \in C$ and $b > 0$ such that for any $\gamma >0$ there exists $p \in C$ with
\[d(x,p) \le b \quad \mbox{and} \quad d(p,Tp) \le \gamma.\]
Then for every $k \in \N$, $g:\mathbb{N} \to \mathbb{N}$,
\[\exists N \le \Phi^+(k,g,L,b,\eta), \forall m \in [N,N+g(N)]\,\, \left(d(x_m,Tx_m) \le \frac{1}{k+1}\right),\]
where
\bua
\Phi^+ & = & h^M(0), \quad h(n)=g(n)+n+1, \quad M=\lceil 3(b+1)/\theta\rceil,\\
\theta & = & \frac{1}{4(k+1)L^2}\eta\left(b+1, \frac{1}{4(k+1)(b+1)}\right).
\eua
If $\eta$ satisfies the extra property from Lemma \ref{lemma-unif-cv-cond(E)}, 
then one can replace it 
by $\tilde{\eta}$ in $\theta.$
\end{theorem}
\begin{proof}
Let $k \in \N$, $g:\mathbb{N} \to \mathbb{N}$. Then there exists $p \in C$ such that $d(x,p) \le b$ and $d(p,Tp) \le 2^{-\Phi^+ - 2}/(3\mu)$. Take $n \le \Phi^+$. Then $d(p,Tp) \le 2^{-n-2}/(3\mu)$. By (\ref{KM-cond(E)-ineq}),
\[d(x_{n+1},p) \le d(x_n,p) + \mu (1-1/L) d(p,Tp) \le d(x_n,p) + 2^{-n-2}/3.\]
Denote $a_n = d(x_n,p)$, $\alpha_0 = 1/6$ and $\alpha_n = \left(1-\sum_{i=0}^{n-1}2^{-i-1}\right)/6$ for $n \ge 1$. Apply \cite[Proposition 6.4]{KohLeu10} with $b_n=\beta_n=\gamma_n=0$, $c_n = 2^{-n-2}/3$, $B_1=B_2=C_2=0$, $A_1=b$, $A_2=1/6$, $C_1=1/6$, $\tilde g(n) = g(n)+1$ to get that for all $n \le \Phi^+ +1$, $d(x_n,p) \le b + 1/6$ and that there exists $N = h^s(0)$ for some $s < M$ such that 
\[\forall i,j \in [N, N+g(N)+1], \quad |a_i-a_j| \le \theta, \quad |\alpha_i - \alpha_j| \le \theta.\]
We show that 
\[\forall m \in [N, N+g(N)], \quad d(x_m,Tx_m) \le \frac{1}{k+1}.\]
Let $m \in [N, N+g(N)]$. Suppose $d(x_m,Tx_m) > 1/(k+1)$. Since $m,m+1 \in [N, N+g(N)+1]$ we have that
\[|d(x_{m+1},p) -d(x_m,p)| \le \theta, \quad |\alpha_{m+1}-\alpha_m| = 2^{-m-2}/3\le \theta.\]
Assume that $d(x_m,p) \ge 1/(4(k+1))$. Note that $m \le N + g(N) < h(N) = h^{s+1}(0) \le h^M(0) = \Phi^+$. Hence, $d(p,Tp) \le 2^{-\Phi^+-2}/(3\mu) < 2^{-m-2}/(3\mu) \le 1/(3\mu)$. Apply (\ref{KM-cond(E)-UCW}) with $\alpha = 1/(4(k+1))$, $\beta=b+2/3$, $\nu = 2^{-m-2}/(3\mu)$ and $\delta = 1/(3\mu)$ to obtain that
\[d(x_{m+1},p) < d(x_m,p) + 2^{-m-2}/3 -2\theta.\]
This yields that $2\theta < d(x_m,p) - d(x_{m+1},p) + 2^{-m-2}/3 \le 2\theta$, a contradiction. So, $d(x_m,p) < 1/(4(k+1))$. Then
\bua
d(x_m,Tx_m) &\le & d(x_m,p) + d(p,Tx_m) \le 2d(x_m,p) + \mu d(p,Tp)\\
& \le & \frac{1}{2(k+1)} + 2^{-m-2}/3 \le \frac{1}{2(k+1)} + \theta \le \frac{1}{k+1}.
\eua
\end{proof}

In the particular case where in the above result the mapping $g=0$, we obtain an approximate fixed point bound for $(x_n)$ in the context of $UCW$-hyperbolic spaces. \\[1mm] 
In case of CAT(0) spaces this bound is quadratic in the error since in 
we then can take $\eta(r,\varepsilon):=\varepsilon^2/8$ and so 
$\tilde{\eta}(r,\varepsilon):=\varepsilon/8.$\\[1mm] Having such an approximate fixed point bound $\Phi$, because $(x_n)$ is additionally uniformly Fej\' er monotone w.r.t. $F$ and $F$ is uniformly closed, we can apply Theorem \ref{finitization-general-GH-Fejer} to obtain, for $C$ convex and totally bounded with II-modulus of total boundedness $\gamma$, a result similar to Theorem \ref{Picard-ne-general-GH-fejer}.(ii) that yields a functional 
$\tilde{\Sigma}:=\tilde{\Sigma}(k,g,\Phi,\gamma,\mu,L)$ with the property that for all $k\in\N$ and all $g:\N\to\N$ there exists $N\le \tilde{\Sigma}$ such that 
\[ \forall i,j\in\! [N,N+g(N)]\,\,\left( d(x_i,x_j)\le \frac1{k+1} \text{~and~}  d(x_i,Tx_i) 
\le \frac1{k+1}\right).\]

\subsection{Mann iteration for asymptotically nonexpansive mappings}

Let $X$ be a $W$-hyerbolic space, $C\se X$  a convex subset and $(k_n)$ be a sequence  in $[0,\infty)$ satisfying $\ds\limn k_n=0$. 

A mapping $T:C\to C$ is said to be {\em asymptotically nonexpansive}  \cite{GoeKir72} with sequence $(k_n)$ if  for all $x,y\in C$ and for all $n\in\N$, 
\[d(T^n x,T^n y) \le (1+k_n)d(x,y).\]

Let $T$ be asymptotically nonexpansive with sequence $(k_n)$ in $[0,\infty)$. We assume furthermore that $(k_n)$ is bounded in sum by some 
$K\in\N$, i.e. $\ds \sum_{n=0}^\infty k_n \le K$. As an immediate consequence, we get that $T$ is Lipschitz continuous with Lipschitz constant $1+K$, hence,
as in the case of strict pseudo-contractions,  it follows that 
 $F$ is uniformly closed with moduli  $\omega_F(k)= (1+K)(4k+4)$ and $\delta_F(k)=2k+1$. 

The  Mann iteration starting with $x\in C$  is defined by 
\beq 
x_0:=x, \quad x_{n+1} :=(1-\lambda_n)x_n+\lambda_n T^n(x_n), \label{deff-Mann-ass-ne}
\eeq 
where $(\lambda_n)$ is a sequence in $\left[\frac1L, 1-\frac1L\right]$ for  some $L\in\N, L\geq 2$.

\begin{lemma}
\be
\item 
For all $n,m\in\N$ and all $p\in C$,
\[
d(x_{n+m},p) \le e^K d(x_n,p) + e^Km(n+m+K)d(p,Tp). 
\]
\item $(x_n)$ is uniformly $(G,H)$-Fej\' er monotone w.r.t. $F$ with modulus 
\[\chi(n,m,r)=m(n+m+K)\lceil e^K\rceil (r+1), \]
where $G(a)=id_{\R^+}$  and $H=e^K id_{\R_+}$. An $H$-modulus is given by $\beta_H(k)=\lceil e^K\rceil(k+1)$.
\ee
\end{lemma}
\begin{proof}
$(i)$ By \cite[Lemma 4.4]{Koh05}. $(ii)$ Apply (i). 
\end{proof}

Effective rates of metastability (and, as a particular case, approximate fixed point bounds) for the Mann iteraton were obtained in \cite{KohLam04} in the setting of uniformly convex Banach spaces and in \cite{KohLeu10} for the more general setting of $UCW$-hyperbolic spaces.  Thus we can apply Theorems \ref{main-step-general-GH-fejer} and \ref{finitization-general-GH-Fejer}. The result of applying Theorem \ref{main-step-general-GH-fejer} gives essentially the rate of metastability that was first extracted in \cite{Koh05} in a more ad-hoc fashion and which now appears as an instance of a general schema for computing rates of metastability. In fact, \cite{Koh05} has been the point of departure of the present paper.

\section{An application  to the Proximal Point Algorithm} \label{maximal-monotone}

The proximal point algorithm is a well-known and popular method employed in approximating a zero of a maximal monotone operator. There exists an extensive literature on this topic which stems from the works of Martinet \cite{Mar70} and Rockafellar \cite{Roc76}. The method consists in constructing a sequence using successive compositions of resolvents which, under appropriate conditions, converges weakly to a zero of the considered maximal monotone operator. If imposing additional assumptions, one can even prove strong convergence. Here we show that we can apply our results to obtain a quantitative version of this algorithm in finite dimensional Hilbert spaces.

\mbox{}

In the sequel $H$ is a real Hilbert space and $A:H \to 2^H$ is a maximal monotone operator. We assume that the set   $\text{zer} A$ of zeros of $A$ is nonempty.   For every $\gamma>0$ let  $J_{\gamma A}= (Id + \gamma A)^{-1}$
be the resolvent of $\gamma A$. Then  $J_{\gamma A}$ is a single-valued firmly nonexpansive mapping defined on $H$ and  $\text{zer} A = Fix(J_{\gamma A})$ for every $\gamma > 0$.  We refer to \cite{BauCom10} for a comprehensive reference on maximal monotone operators.

Let $x_0 \in H$ and $(\gamma_n)$ be a sequence in $(0,\infty)$. The proximal point algorithm starting with $x_0\in H$ is defined as follows:
\[x_{n+1} = J_{\gamma_n A}x_n.\]

Let us take $F:=\text{zer} A$. One can easily see that 
\[
F =  \bigcap_{k\in\N}\tilde{F}_k, \quad \text{where } \tilde{F}_k=\bigcap_{i \le k}\left\{ x\in H \mid \|x - J_{\gamma_i A}x\|\le \frac{1}{k+1}\right\}.
\]
and that  $AF_k=\tilde{F}_k$ for every $k \in \N$. 

\details{$\subseteq$ If $x\in F$, then $x - J_{\gamma_i A}x=0$ for all $i$, hence $x\in  \tilde{F}_k$ for all $k$.
$\supseteq$  Assume that $x\in \tilde{F}_k$ for all $k$. Then $\|x - J_{\gamma_1 A}x\|\leq \frac{1}{k+1}$ for all $k$. Hence we must have $x\in Fix(J_{\gamma_1 A})=\text{zer} A$.\\
}
Furthermore,  $F$ is uniformly closed with moduli $\omega_F(k)=4k+3, \,\delta_F(k)=2k+1$. 

\details{
Let $k \ge 0$ and $x,y\in H$ with $\|x-J_{\gamma_i A}x\| \le 1/(2k+2)$ for every $i \le 2k+1$ and $\|x-y\| \le 1/(4k+4)$. Then, for every $i \le k$,
\bua
\|y - J_{\gamma_i A}y\| & \le & \|x-y\| + \|x - J_{\gamma_i A}x\| + \|J_{\gamma_i A}x - J_{\gamma_i A}y\|\\
& \le & 2\|x-y\| + \|x - J_{\gamma_i A}x\| \le \frac{1}{2k+2} + \frac{1}{2k+2} = \frac{1}{k+1}.
\eua
Thus, $y \in AF_{k}$.
}

\begin{lemma}\label{lemma-xn-uFm-PPA}
\be
\item For all $n\in\N, m\in \N^*$ and all  $p\in H$, 
\beq 
\|x_{n+m}-p\|\le \|x_n-p\| + \sum_{i=n}^{n+m-1}\|p-J_{\gamma_i A}p\|. \label{ppa-ineq}
\eeq
\item  $(x_n)$ is uniformly Fej\'{e}r monotone w.r.t. $F$ with modulus $\chi(n,m,r)=\max\{n+m-1,m(r+1)\}$.
\ee
\end{lemma}
\begin{proof}
\be
\item Remark that 
\bua
\|x_{n+1}-p\| &=& \| J_{\gamma_n A}x_n-p\|\leq \| J_{\gamma_n A}x_n- J_{\gamma_n A}p\|+\| J_{\gamma_n A}p-p\| \\
&\le &  \| x_n-p\|+\| J_{\gamma_n A}p-p\|
\eua
and use induction.
\item Apply \eqref{ppa-ineq} and the fact that $p\in AF_{\chi(n,m,r)}$ implies  that for all  $m\geq 1$ and all $l\leq m$,
\[\sum_{i=n}^{n+l-1}\|p-J_{\gamma_i A}p\| \leq \sum_{i=n}^{n+m-1}\|p-J_{\gamma_i A}p\| \leq \frac{m}{\chi(n,m,r)+1} < \frac{1}{r+1}.\]

\details{Let $r,n,m \in \N$ and  denote for simplicity $k:=\chi(n,m,r)$. The case $m=0$ is trivial, hence we can consider $m \ge 1$. Assume that $p\in AF_k$. 
Then for every $i \le n+m-1 \le k$, $\|x-J_{\gamma_iA}x\| \le 1/(k+1)$. Hence, for all $l \le m$,
\[\|x_{n+l}-x\| \le \|x_n-x\| + \sum_{i=n}^{n+m-1}\|x-J_{\gamma_i A}x\| \le \|x_n-x\| + \frac{m}{k+1}.\]
Thus, for every $m \in \N$ and $l \le m$,
\[
\|x_{n+l}-x\| \le \|x_n-x\| + \frac{m}{k+1} \le \|x_n-x\| + \frac{1}{r+1} \quad \text{since } k \ge m(r+1).
\]
}
\ee
\end{proof}

In the following we consider for $n \in \N$, 
$$u_n = \frac{x_n-x_{n+1}}{\gamma_n}.$$

The next lemma is well-known.  We refer, e.g., to  the proof of \cite[Theorem 23.41]{BauCom10} and \cite[Exercise 23.2, p. 349]{BauCom10}.

\begin{lemma}
\be
\item For every $p\in \text{zer} A$ and every $n,i \in \N$,
\bea
\|x_{n+1}-p\|^2  & \leq &  \|x_n-p\|^2-\|x_n - x_{n+1}\|^2 \label{ineq-xn+1-xn} \\
\|J_{\gamma_n A}x_n - J_{\gamma_i A}x_n\| &\le &  |\gamma_n - \gamma_i|\frac{\|x_n - x_{n+1}\|}{\gamma_n} \label{ineq-Jgi-Jgn}\\
\|x_n - J_{\gamma_i A}x_n\| &\le & \|x_n - x_{n+1}\|  + |\gamma_n - \gamma_i|\frac{\|x_n - x_{n+1}\|}{\gamma_n}.\label{ineq-xn-Jgi}
\eea
\item The sequence $(\|u_n\|)$ is nonincreasing. 
\ee
\end{lemma}

\details{Let $p\in \text{zer} A$. Since $J_{\gamma_n A}$ is firmly nonexpansive, we get that for all $n\in \N$,
\bua
\|x_{n+1}-p\|^2  & =  &\|J_{\gamma_n A}x_n-J_{\gamma_nA}p\|^2\\
& \le & \|x_n-p\|^2 -\| (id_H-J_{\gamma_nA})x_n-(id_H-J_{\gamma_nA})p\|^2\\
& = &\|x_n-p\|^2 -\| x_n-x_{n+1}\|^2.
\eua
For \eqref{ineq-Jgi-Jgn} see, for example, \cite[Exercise 23.2, page 349]{BauCom10}.

For \eqref{ineq-xn-Jgi}, note that 
\[\|x_n - J_{\gamma_i A}x_n\| \le \|x_n - x_{n+1}\|  + \|J_{\gamma_n A}x_n - J_{\gamma_i A}x_n\|\]
and apply \eqref{ineq-Jgi-Jgn}.\\
}

\begin{lemma}\label{main-lemma-xn-un}
Assume that $\ds\sum_{i=0}^\infty \gamma_n^2 = \infty$ with a rate of divergence $\theta$ and that $b > 0$ is an upper bound on $\|x_0 - p\|$ for some  $p\in  \text{zer} A$.
Then
\be
\item\label{mod-linf}  
$\linfn \|x_n-x_{n+1}\|=0$ with modulus of liminf 
$$\Delta(k,L,b):=\left\lceil b^2 (k+1)^2\right\rceil + L - 1, \ \mbox{i.e.}$$ 
for every $k \in \N$  and $L\in\N$  there exists $L\leq N \le \Delta(k,L,b)$ such that 
$\|x_n-x_{n+1}\|\le 1/(k+1)$.
\item\label{un-rate} $\limn u_n=0$ with rate of convergence  
$\beta(k,\theta,b):=\theta\left(\left\lceil b^2 (k+1)^2\right\rceil\right)$.
\ee
\end{lemma}
\begin{proof}

\be
\item Applying \eqref{ineq-xn+1-xn}  repeatedly we  get that
\[\sum_{n = 0}^\infty \|x_n - x_{n+1}\|^2\leq \|x_0-p\|^2\leq b^2.\]

Let $k,L \in \N$ and $\Delta := \Delta(k,L,b)$. Suppose that for every $L \le n \le \Delta$, $\|x_n - x_{n+1}\| > 1/(k+1)$. Then 
\[\left(\Delta - L + 1\right) \frac{1}{(k+1)^2} < \sum_{n=L}^{\Delta} \|x_n - x_{n+1}\|^2 \le b^2,\]
which yields $\Delta < b^2(k+1)^2 + L - 1$, a contradiction.
\item Let $k \in \N$ and $\beta := \beta(k,\theta,b)$. Since $(\|u_n\|)$ is nonincreasing, it is enough to show that there exists $0 \le N \le \beta$ such that $\|u_N\| \le 1/(k+1)$. Suppose that for every $0 \le n \le \beta$, $\|u_n\| > 1/(k+1)$. Then
\bua
\frac{1}{(k+1)^2}\left\lceil b^2 (k+1)^2\right\rceil  &\le& \frac{1}{(k+1)^2}\sum_{n=0}^{\beta} \gamma_n^2 < \sum_{n = 0}^\beta \gamma_n^2\|u_n\|^2\\
&=& \sum_{n = 0}^\beta \|x_n - x_{n+1}\|^2\le b^2.
\eua
We have obtained a contradiction. 
\ee
\end{proof}

\begin{theorem}\label{lemma-xn-afp-PPA} 
Assume that $\ds\sum_{i=0}^\infty \gamma_n^2 = \infty$ with a rate of divergence $\theta$. Then  $(x_n)$ has approximate $F$-points with an approximate $F$-point bound 
\[\Phi(k,m_k,\theta,b) := \theta\left(\left\lceil b^2 (M_k+1)^2\right\rceil\right)\left\lceil b^2 (M_k+1)^2\right\rceil - 1,\]
where $\ds m_k = \max_{0 \le i \le k} \gamma_i$ and $M_k = \left\lceil (k+1)(2+m_k) \right\rceil - 1$ and $b>0$ is such that $b\geq \|x_0 - p\|$ for some $p\in  \text{zer} A$.
\end{theorem}
\begin{proof}
Let $k\in\N$. By Lemma \ref{main-lemma-xn-un}.\eqref{mod-linf}, there exists $N_1\leq \Delta (M_k,0,b)$ 
such that 
$$\|x_{N_1} - x_{N_1+1}\| \le \frac{1}{M_k + 1} \le \frac{1}{(k+1)(2+m_k)}.$$ 
If $\gamma_{N_1}\geq 1$, it follows by  \eqref{ineq-xn-Jgi} that for  all $i \le k$,
\bua
\|x_{N_1} - J_{\gamma_i A}x_{N_1}\| & \le & \|x_{N_1} - x_{N_1+1}\|  + |\gamma_{N_1} - \gamma_i|\frac{\|x_{N_1} - x_{N_1+1}\|}{\gamma_{N_1}}\\
& \leq & \left(2 + \frac{\gamma_i}{\gamma_{N_1}}\right)\|x_{N_1} - x_{N_1+1}\| \\
&\le & (2 + m_k)\frac{1}{(k+1)(2+m_k)} = \frac{1}{k+1}.
\eua
Assume that $\gamma_{N_1}< 1$. Apply again Lemma \ref{main-lemma-xn-un}.\eqref{mod-linf} to get the  existence  of $N_2\leq \Delta (M_k,N_1+1,b)$ such that $N_2>N_1$ and
$\|x_{N_2} - x_{N_2+1}\| \le \frac1{M_k+1}$. If $\gamma_{N_2}\geq 1$ we use again the above argument. If $\gamma_{N_2}< 1$, we apply once more the fact that $\Delta$ is a modulus of liminf for
$\|x_n-x_{n+1}\|$.  Let us denote for simplicity $\beta:= \theta\left(\left\lceil b^2 (M_k+1)^2\right\rceil\right)$.
Applying this argument $\beta$ times  we get a finite sequence $N_1<N_2<\ldots < N_\beta$ such that either $\gamma_{N_j}\geq 1$  for some $j$ or $\gamma_{N_j}<1$ for all
$j=1,\ldots, \beta$. In the first case, we have as above  that $\|x_{N_j} - J_{\gamma_iA}x_{N_j}\|\leq  \frac{1}{k+1}$ for all $i\leq k$. In the second case, since $N_\beta\geq \beta$ and $(\|u_n\|)$ is nonincreasing, an application of Lemma \ref{main-lemma-xn-un}.\eqref{un-rate} for $M_k$ gives us
\[\|u_{N_\beta}\|\leq \|u_\beta\| \leq \frac1{M_k+1}\leq \frac{1}{(k+1)(2+m_k)}.\]

It follows then that for all $i\leq k$,
\bua
\|x_{N_\beta} - J_{\gamma_i A}x_{N_\beta}\| & \le &  \|x_{N_\beta} - x_{N_\beta+1}\|  + |\gamma_{N_\beta} - \gamma_i|\frac{\|x_{N_\beta} - x_{N_\beta+1}\|}{\gamma_{N_\beta}}\\
&=& \gamma_{N_\beta}\|u_{N_\beta}\|+|\gamma_{N_\beta} - \gamma_i|  \|u_{N_\beta}\| \leq  \left(2\gamma_{N_\beta} + \gamma_i\right)\|u_\beta\|\\
& \le &  (2 +m_k)\frac{1}{(k+1)(2+m_k)}= \frac{1}{k+1}.
\eua

Since $N_1\leq \Delta (M_k,0,b)=\left\lceil b^2 (M_k+1)^2\right\rceil-1$ and for all $j= 2,\ldots, \beta$, 
$$N_j \le \Delta(M_k,N_{j-1}+1,b) =\left\lceil b^2 (M_k+1)^2\right\rceil + N_{j-1}$$
we get that $N_\beta \le \beta\left\lceil b^2 (M_k+1)^2\right\rceil - 1$,  which finishes the proof.
\end{proof}

As an immediate consequence of Proposition \ref{general-GH-fejer} and
Remark \ref{general-fejer-boundedly-compact}
we obtain the well-known fact that in $\R^n$, under the hypothesis that
$\ds\sum_{i=0}^\infty \gamma_n^2 = \infty$, the proximal point algorithm
converges strongly to a zero of the maximal monotone operator $A$.
Furthermore,  since $\|x_n\| \leq  M:=b+\|p\|$ (where $b,p$ are as above)
and, by Example \ref{ex-bd-set-Rn}, $\ol{B}(0,M)$ is totally bounded with
II-modulus $\gamma(k) = \left\lceil2(k+1)\sqrt{n}M\right\rceil^n$,  we can
apply the quantitative Theorem \ref{finitization-general-GH-Fejer} to get
rates of metastability for $(x_n)$.

\mbox{}

\noindent
{\bf Acknowledgements:} \\[1mm] 
Ulrich Kohlenbach was supported by the German Science Foundation (DFG 
Project KO 1737/5-2).\\[1mm]
Lauren\c tiu Leu\c stean was supported by a grant of the Romanian 
National Authority for Scientific Research, CNCS - UEFISCDI, project 
number PN-II-ID-PCE-2011-3-0383. \\[1mm]
Adriana Nicolae was supported by a grant of the Romanian
Ministry of Education, CNCS - UEFISCDI, project number PN-II-RU-PD-2012-3-0152.

\end{document}